\documentclass[preprint]{elsarticle}
\usepackage{graphicx} 
\usepackage{amsmath}
\usepackage{amsfonts}
\usepackage{amssymb}
\usepackage{xcolor}
\usepackage{hyperref}
\usepackage{cleveref}
\usepackage{todonotes}
\usepackage{caption}
\usepackage{subcaption}

\usetikzlibrary{decorations.pathreplacing}

\usepackage{amsthm} 
\newtheorem{theorem}{Theorem}[section]

\newtheorem{lemma}[theorem]{Lemma}
\newtheorem{remark}{Remark}[section]

\makeatletter
\newcommand{\opnorm}{\@ifstar\@opnorms\@opnorm}
\newcommand{\@opnorms}[1]{%
  \left|\mkern-1.5mu\left|\mkern-1.5mu\left|
   #1
  \right|\mkern-1.5mu\right|\mkern-1.5mu\right|
}
\newcommand{\@opnorm}[2][]{%
  \mathopen{#1|\mkern-1.5mu#1|\mkern-1.5mu#1|}
  #2
  \mathclose{#1|\mkern-1.5mu#1|\mkern-1.5mu#1|}
}
\makeatother

\newcommand{\matr}[1]{\mathbf{#1}}

\newcommand{\local}{L}
\newcommand{\nonlocal}{N}
\newcommand{\interactionSolve}{{\nonlocal, \mathcal{I}}}
\newcommand{\interactionGiven}{{\nonlocal, \tilde{\mathcal{I}}}} 
\newcommand{\globalNonlocal}{\mathring{\nonlocal}}
\newcommand{\localBoundSolve}{\Gamma}
\newcommand{\localBoundGiven}{\tilde \Gamma} 

\newcommand{\is}[1][]{\mathcal{I}_{#1}} 
\newcommand{\mesh}[1][]{\mathcal{T}_{#1}} 

\newcommand{\lapl}[1][]{\matr A_{#1}} 
\newcommand{\restr}[1][]{\matr R_{#1}} 

\renewcommand{\tilde}[1]{\widetilde{#1}}

\newcommand{\abs}[1]{\left|#1\right|}
\newcommand{\norm}[1]{\left|\!\left|#1\right|\!\right|}

\renewcommand{\vec}[1]{\boldsymbol{#1}}

\begin{document}

\title{A Splice Method for Local-to--Nonlocal Coupling of Weak Forms}
\author[snl]{Shuai Jiang}
\ead{sjiang@sandia.gov}

\author[snl]{Christian Glusa}
\ead{caglusa@sandia.gov}

\affiliation[snl]{organization={Center for Computing Research, Sandia National Laboratories}, city={Albuquerque}, postcode={87321}, state={NM}, country={USA}}

\begin{abstract}
  We propose a method to couple local and nonlocal diffusion models.
  By inheriting desirable properties such as patch tests, asymptotic compatibility and unintrusiveness from related splice and optimization-based coupling schemes, it enables the use of weak (or variational) formulations, is computationally efficient and straightforward to implement.
  We prove well-posedness of the coupling scheme and demonstrate its properties and effectiveness in a variety of numerical examples.
\end{abstract}


\maketitle

\section{Introduction}

Nonlocal models are characterized by integral equations with a kernel function that captures long range interactions.
The increased flexibility encapsulated in the kernel allows one to capture effects that classical models using partial differential equations cannot reproduce in general, without the use of multiscale coefficients.
Moreover, nonlocal models allow one to effectively handle low regularity solutions without a need for additional mechanisms such as tracking of discontinuities.
Commonly used nonlocal descriptions are diffusion operators \cite{DuGunzburgerEtAl2012_AnalysisApproximationNonlocalDiffusion} of the form
\begin{align*}
  \mathcal{L}u(\vec{x})=\int (u(\vec y) - u(\vec x)) \gamma(\vec x, \vec y) \, d{y} ,&& \text{with kernel }\gamma \text{ and } \vec{x},\vec{y}\in\mathbb{R}^{d},
\end{align*}
and the peridynamic descriptions of mechanics~\cite{Silling2000_ReformulationElasticityTheoryDiscontinuities}.
Similar to local equations, a wide variety of discretization schemes such as finite elements, finite differences or mesh-free/particle methods are available~\cite{DEliaDuEtAl2020_NumericalMethodsNonlocalFractionalModels}.

\emph{Local-to-Nonlocal} (LtN) coupling aims to combine a nonlocal model posed on a sub-region of the computational domain with a local model that is prescribed on the complement.
There are several compelling reasons for use of a LtN coupling method:
\begin{itemize}
\item
  \textbf{Computational efficiency:}
  Since the assembly and solution of local models is often significantly less expensive in both computational time and in memory usage, using a local model instead of a nonlocal one in part of the simulation domain can lead to appreciable savings.
  This is especially useful when nonlocal effects are known to be more pronounced in sub-regions, e.g. around a fracture in a peridynamic simulation.
  In this scenario, the goal is to recover the solution of a nonlocal equation posed on the entire domain as accurately as possible at a fraction of the computational cost.

\item
  \textbf{Avoidance of nonlocal volume conditions:}
  It is a non-trivial task to derive nonlocal volume conditions that capture the same behavior as classical boundary conditions would for a local model.
  Moreover, it is often not straightforward to decide what data should be prescribed via a volume condition, whereas conditions posed on surfaces are much less problematic.
  The issue can be mitigated by LtN coupling where the nonlocal model is replaced with a local model in a sufficiently thick layer adjacent to the domain boundary.
  Then, instead of imposing volume conditions, the same conventional boundary conditions are imposed as for a classical local description.

\item
  \textbf{Multi-phase or multi-physics problems:}
  Coupling problems involving local and nonlocal models arise when different types of physics govern different parts of the domain.
  Similarly, in settings where local and nonlocal description of the same physics are more appropriate in separate parts of the domain coupling schemes are required.

\end{itemize}

A wide range of LtN coupling methods is available in the literature.
A review of techniques can be found in \cite{d2021review} and the references therein.
LtN methods can be classified based on whether adjustments are made to the kernel function in the transition between local and nonlocal regions.
In particular, several methods have been proposed that rely on adjusting the interaction horizon of the kernel.
On the other hand, and more pertinent for the present work, are methods that employ a fixed kernel function such as splice coupling~\cite{silling2015variable} or optimization-based coupling~\cite{d2016coupling,DEliaBochev2021_FormulationAnalysisComputationOptimization}.
Commonly, the following properties are deemed desirable for LtN coupling methods~\cite{d2021review}:
\begin{itemize}
\item
  \textbf{Patch tests:}
  If a fully nonlocal model and a fully local model have matching solutions, a coupling of the two models should also admit the same solution.

\item
  \textbf{Asymptotic compatibility:}
  Nonlocal operators often recover classical local models in certain parameter limits, such as vanishing interaction horizon~\cite{TianDu2014_AsymptoticallyCompatibleSchemesApplications}.
  It is therefore appropriate to require that the same should also hold for a coupled model.

\item
  \textbf{Energy equivalence:}
  The coupling problem can be equipped with an energy that is equivalent to the energies of fully local and fully nonlocal problems.

\item
  \textbf{Computational cost:}
  In order to realize computational savings in replacing a fully nonlocal models by a LtN coupling, one would like that the computational load of assembling and solving the coupled model is roughly equivalent to the convex combination of the cost of full local and nonlocal models, weighted by the fraction of the computational domain assigned to each.

\item
  \textbf{Unintrusiveness:}
  While it is unavoidable that one has to assemble both local and nonlocal models over sub-domains, it is helpful to keep coupling-specific implementation details to a minimum.
  Ideally, the coupling method should not be difficult to implement and is discretization agnostic.
  Specialized kernel functions, hyper-parameters or weights~\cite{LiLu2017_QuasiNonlocalCouplingNonlocalDiffusions} that need to be carefully tuned can adversely affect the applicability of a coupling scheme.

\end{itemize}

In the present work, we propose a splice method that is closely related to optimization-based coupling.
In contrast to splice methods found in the literature which are tied to mesh-free discretizations of the nonlocal operator in strong form, our method couples weak forms.
This means that the complete analysis framework for local and nonlocal equations is at our disposal as well as the use of finite element discretizations for both local and nonlocal problem.
Additionally, the proposed scheme significantly reduces the computational cost when compared with optimization-based coupling approaches as no iterative minimization procedure is required.
Finally, the method inherits several of the desirable properties listed above from optimization and splice methods.

While we describe the method in the setting of nonlocal diffusion problems, its extension to vector-valued descriptions such as peridynamics is straightforward.

In \Cref{sec:problem-definition} we give a description of nonlocal and local models used in the coupling and their respective discretization using finite elements.
The coupling of the different discretizations is outlines in \Cref{sec:local-nonl-coupl}.
In \Cref{sec:prop-splice-ltn} we give results regarding well-posedness of the coupling and a proof that patch tests are satisfied by the method.
Finally, in \Cref{sec:numerical-results} we show numerical results in 1D and 2D that illustrate the effectiveness of the proposed approach.

\section{Problem Definition}
\label{sec:problem-definition}

\subsection{Continuous Problem}

Let $\Omega \subset \mathbb{R}^d$ be a bounded open domain. 
We will exclusively focus on $d = 1, 2$ for the examples of the paper, though the methods presented generalize to higher dimensions.
Let $\gamma(\vec x, \vec y): \mathbb{R}^d \times \mathbb{R}^d \to \mathbb{R}^d$ be a non-negative, symmetric kernel function which is potentially singular at $\vec x = \vec y$.
We assume that $\gamma$ has horizon $\delta \in(0,\infty)$ meaning that $\gamma(\vec x, \vec y) = 0$ if $\norm{\vec x - \vec y} > \delta$ where $\norm{\cdot}$ is the standard $\ell^2$ norm.
The nonlocal diffusion operator is defined as, for all $\vec x \in \Omega$,
\begin{align}\label{eqn:nonlocal-diff}
    \mathcal{L}_\nonlocal u(\vec x) := \operatorname{p.v.}\int_{\mathbb{R}^d} (u(\vec y) - u(\vec x)) \gamma(\vec x, \vec y) \, d{y}.
\end{align}
Here, $\operatorname{p.v.}$ signifies that the integral might have to be taken in the \emph{principal value} sense, depending on the strength of singularity of the kernel $\gamma$.

In the present work we will consider two families of kernel functions:
\begin{itemize}
\item
fractional kernels with \emph{fractional order} \(s\in(0,1)\)
\begin{align}
  \gamma_{F}(\vec{x},\vec{y})
  &:= C_{F}(d,s,\delta) \norm{\vec{x}-\vec{y}}^{-d-2s}\chi_{\norm{\vec{x}-\vec{y}}<\delta} \label{eq:fracKernel}
    \intertext{where \(\chi\) is the indicator function and}
    C_{F}(d,s,\delta)
  &=
    \begin{cases}
      \frac{2^{2s}s\Gamma(s+d/2)}{\pi^{d/2}\Gamma(1-s)}& \text{if } \delta=\infty, \\
      \frac{(2-2s)\delta^{2s-2}d\Gamma(d/2)}{\pi^{d/2}}& \text{if } \delta<\infty,
    \end{cases}\label{eq:fracKernelNormalization}
\end{align}
\item
  integrable kernels with singularity of strength \(0\leq\alpha<d\), \(\delta<\infty\)
  \begin{align}
    \gamma_{I}(\vec{x},\vec{y};\alpha,\delta)
    &:= C_{I}(d,\alpha,\delta) \norm{\vec{x}-\vec{y}}^{-\alpha}\chi_{\norm{\vec{x}-\vec{y}}<\delta} \label{eq:integrableKernel}
      \intertext{where}
      C_{I}(d,\alpha,\delta) &= \frac{(d+2-\alpha) d\Gamma(d/2)}{\pi^{d/2}\delta^{d+2-\alpha}}.\label{eq:integrableKernelNormalization}
\end{align}
In particular we call the kernel with \(\alpha=0\) the \emph{constant kernel} and the kernel with \(\alpha=1\) the \emph{inverse distance kernel}.
\end{itemize}
It should be noted that the presented methodology has rather weak assumptions on the properties of the kernel function and can therefore readily be generalized to other types of nonlocal problems.

Let $\Omega_{\mathcal{I}}$ be the $\delta$-collar around $\Omega$ where $\gamma$ is non-zero
\begin{align*}
    \Omega_{\mathcal{I}} := \{\vec x \in  \mathbb{R}^{d}\setminus\Omega \mid \exists \vec y \in \Omega \text{ s.t. } \gamma(\vec x, \vec y) \not = 0 \}.
\end{align*}
The domain $\Omega_\mathcal{I}$ serves to support a volumetric boundary conditions (otherwise known as the interaction region) that is typical of nonlocal models.
We define the associated energy norm $\opnorm{u}_{\Omega \cup \Omega_\mathcal{I}}$ where
\begin{align*}
  \opnorm{u}^{2}_{\omega} = \iint_{\omega^{2}} (u(\vec x) - u(\vec y))^2 \gamma(\vec x, \vec y) \, d{y} \, dx,
\end{align*}
with the energy space of $V(\Omega\cup\Omega_\mathcal{I})$ where
\begin{align*}
  V(\omega) :=\{v\in L^{2}(\omega) \mid \opnorm{v}_{\omega} < \infty\}
\end{align*}
for $\omega$ some open set. 

We are interested in the solution of nonlocal Poisson problem of the form
\begin{equation}\label{eqn:nonlocal-model-full}
  \left\{
    \begin{aligned}
    -\mathcal{L}_\nonlocal u(\vec x) &= f(\vec x)  && \vec{x} \in \Omega, \\
                           u(\vec x) &= g(\vec x) && \vec{x} \in \Omega_{\mathcal{I}},
    \end{aligned}
    \right.
\end{equation}
where $f \in L^2(\Omega)$ and $g \in \operatorname{trace}_{\Omega_{\mathcal{I}}}V(\Omega \cup \Omega_\mathcal{I})$ with 
\begin{align}\label{eqn:nonlocal-trace-def}
  \operatorname{trace}_{\omega} \mathcal H := \{v \in L^2(\omega) \mid \exists u \in \mathcal H, u|_{\omega} = v \}.
\end{align}
The choice of kernel endows the nonlocal operator in \cref{eqn:nonlocal-diff} with several properties.
Among others it determines the regularity lifting of the above Poisson problem.

\Cref{eqn:nonlocal-model-full} generally corresponds to a linear system which is difficult to assemble and solve, owing to the nonlocal nature of the operator \cref{eqn:nonlocal-diff}.
We strive to reduce the computational complexity by replacing $\mathcal{L}_\nonlocal$ with a cheaper surrogate operator.
To do so, we assume two open sets $\Omega_\local, \Omega_\nonlocal \subset \Omega$ such that
\begin{align}\label{eqn:domain-assumption}
  \overline \Omega_\nonlocal \cup \overline \Omega_\local = \overline\Omega.
\end{align}
On $\Omega_\nonlocal$ the nonlocal equation involving \(\mathcal{L}_{\nonlocal}\) is solved, while a classical local model is used on $\Omega_\local$.
We adopt the convention that a subscript of $\cdot_\nonlocal$ or $\cdot_\local$ signify nonlocal and local respectively. 
Naturally, one is inclined to choose subdomains such that the intersection \(\Omega_{\nonlocal}\cap\Omega_{\local}\) is as small as possible in order to avoid unnecessary computation.
Ideally we would like to choose the two subdomains to be disjoint, but it will be apparent that using an overlap of \(\mathcal{O}(h)\) is preferable in the discrete setting with
\(h\) being the mesh size.
In the continuum limit of \(h\rightarrow0\) this corresponds to the disjoint case, yet it will be apparent very soon that the continuum problem is significantly harder to discuss than the discrete case.

Since nonlocal equations rely on boundary values defined in a $\delta$-collar around $\Omega_\nonlocal$, we define $\Omega_\interactionSolve \subset \Omega$ 
\begin{align*}
    \Omega_\interactionSolve = \{\vec x \in  \Omega\setminus\Omega_{\nonlocal} \mid \exists \vec y \in \Omega_\nonlocal \text{ s.t. } \gamma(\vec x, \vec y) \not = 0 \}
\end{align*}
be the open set of the nonlocal boundary condition which also lies in the local region.
Similarly, let
\begin{align*}
    \Omega_{\interactionGiven} =  \{\vec x \in  \mathbb{R}^{d}\setminus\Omega \mid \exists \vec y \in \Omega_\nonlocal \text{ s.t. } \gamma(\vec x, \vec y) \not = 0\}
\end{align*}
be the remaining interaction region corresponding to a prescribed, Dirichlet boundary condition. 
Thus our nonlocal model posed on the subdomain \(\Omega_{\nonlocal}\) is
\begin{equation}\label{eqn:nonlocal-model}
  \left\{
    \begin{aligned}
      -\mathcal{L}_N u(\vec x) &= f(\vec x), && \vec x \in \Omega_{\nonlocal}, \\
      u_N(\vec x) &= u_L(\vec x), && \vec x \in \Omega_{\interactionSolve}, \\
      u_N(\vec x) &= g(\vec x),  && \vec x \in\Omega_{\interactionGiven},
    \end{aligned}
    \right.
\end{equation}
where $u_L$ the solution on the local domain defined below.

We assume that the local model is simply the usual integer-order Laplacian $\mathcal{L}_L = \Delta$.
This is motivated by the fact that in the local limit $\delta\rightarrow 0$ the nonlocal operator $\mathcal{L}_\nonlocal$ recovers the local one, assuming appropriate normalization, as for example in \eqref{eq:fracKernelNormalization}, \eqref{eq:integrableKernelNormalization}.
Let $\localBoundSolve = \partial \Omega_L \cap \Omega$ be the open set corresponding to the boundary of $\Omega_\local$ which intersects $\Omega_\nonlocal$, and let $\localBoundGiven = \partial \Omega_\local \setminus \localBoundSolve$ be the remaining boundary which is defined by a given Dirichlet boundary condition.
The model on $\Omega_\local$ is thus
\begin{equation}\label{eqn:local-model}
  \left\{
\begin{aligned}
    -\mathcal{L}_L  u_L(\vec x) &= f(\vec x), && \vec x \in \Omega_{\local}, \\
     u_L(\vec x) &=  u_N(\vec x), && \vec x \in \localBoundSolve, \\
     u_L(\vec x) &= g(\vec x), && \vec x \in \localBoundGiven.
\end{aligned}
\right.
\end{equation}
We assume that in addition to previous regularity assumptions, the boundary condition satisfies $g \in H^{1/2}(\localBoundGiven)$.

In what follows, we consider equations \cref{eqn:nonlocal-model} and \cref{eqn:local-model} as the strong form of a local-nonlocal coupling problem.
The system of equations is difficult to analyze in the continuous case, and may not even be properly defined.
For example, in \cref{eqn:local-model}, there is a question of compatible boundary conditions with $u_N$ and $g$.
Moreover, it is not at all clear whether $u_\nonlocal$ has the required regularity as boundary data for a classical local Poisson problem.
We will forgo an analysis of the problem at the continuous level and instead limit ourselves to a discussion of the discrete setting and prove well-posedness in that context only.

We illustrate the various subdomains for two different geometric configurations in \cref{fig:simple-domain} in both one and two dimensions.
For the purposes of this illustration, the overlap between \(\Omega_{\nonlocal}\) and \(\Omega_{\local}\) is chosen to be empty.
In the first configuration, we split the domain into a local subdomain on the left and a nonlocal subdomain on the right.
The nonlocal problems have a Dirichlet volume condition enforced on a non-empty region while the local problem resembles the usual Poisson problem. 
The second configuration illustrates the case of a nonlocal inclusion in a local domain.
In this case \(\Omega_{\interactionGiven}=\emptyset\) provided that the horizon \(\delta\) is small enough.

\begin{figure}
    \centering
    \begin{subfigure}[b]{0.48\textwidth}
      \centering
      \begin{tikzpicture}[scale=.8]
    \draw[very thick] (0,0) -- (6.5,0);
    \draw[draw, blue, double=blue, double distance=2\pgflinewidth](0,3pt) -- (0,-3pt); 
    \draw[draw, red, double=red, double distance=2\pgflinewidth](3,3pt) -- (3,-3pt); 
    \draw (3,3pt) -- (3,-3pt);
    \draw (0,3pt) -- (0,-3pt);
    \draw[densely dotted] (2.5,4pt) -- (2.5,-4pt);
    \draw[densely dotted] (6,  4pt) -- (6,  -4pt);
    \draw[densely dotted] (6.5,4pt) -- (6.5,-4pt);

\draw [decorate,decoration={brace,amplitude=2pt,raise=2ex}]
  (2.5,0) -- (3,0) node[midway,yshift=2em]{$\Omega_{\interactionSolve}$};
  \draw [decorate,decoration={brace,amplitude=2pt,raise=2ex}]
  (6,0) -- (6.5,0) node[midway,yshift=2em]{$\Omega_{\interactionGiven}$};
\draw [decorate,decoration={brace,amplitude=5pt,mirror,raise=2ex}]
  (0,0) -- (3,0) node[midway,yshift=-2em]{$\Omega_\local$};
  \draw [decorate,decoration={brace,amplitude=5pt,mirror,raise=2ex}]
  (3,0) -- (6,0) node[midway,yshift=-2em]{$\Omega_\nonlocal$};
\end{tikzpicture}
    \end{subfigure}
        \begin{subfigure}[b]{0.48\textwidth}
      \centering
      \begin{tikzpicture}[scale=.8]
    \draw[very thick] (0,0) -- (6,0);
    \draw[draw, blue, double=blue, double distance=2\pgflinewidth](0,3pt) -- (0,-3pt); 
    \draw[draw, blue, double=blue, double distance=2\pgflinewidth](6,3pt) -- (6,-3pt); 
    \draw[draw, red, double=red, double distance=2\pgflinewidth](3.5,3pt) -- (3.5,-3pt); 
    \draw[draw, red, double=red, double distance=2\pgflinewidth](2.5,3pt) -- (2.5,-3pt); 
    \draw (6,3pt) -- (6,-3pt);
    \draw (0,3pt) -- (0,-3pt);
    \draw[densely dotted] (2.5,4pt) -- (2.5,-4pt);
    \draw[densely dotted] (3.5,  4pt) -- (3.5,  -4pt);
    \draw[densely dotted] (4.5,  4pt) -- (4.5,  -4pt);
    \draw[densely dotted] (1.5,  4pt) -- (1.5,  -4pt);

\draw [decorate,decoration={brace,amplitude=2pt,raise=2ex}]
  (1.5,0) -- (2.5,0) node[midway,yshift=2em]{$\Omega_{\interactionSolve}$};
  \draw [decorate,decoration={brace,amplitude=2pt,raise=2ex}]
  (3.5,0) -- (4.5,0) node[midway,yshift=2em]{$\Omega_{\interactionSolve}$};
    \draw [decorate,decoration={brace,amplitude=2pt,raise=2ex}]
  (2.5,0) -- (3.5,0) node[midway,yshift=2em]{$\Omega_\nonlocal$};
\draw [decorate,decoration={brace,amplitude=5pt,mirror,raise=2ex}]
  (0,0) -- (2.5,0) node[midway,yshift=-2em]{$\Omega_\local$};
\draw [decorate,decoration={brace,amplitude=5pt,mirror,raise=2ex}]
  (3.5,0) -- (6,0) node[midway,yshift=-2em]{$\Omega_\local$};
\end{tikzpicture}
    \end{subfigure}
\\

    \begin{subfigure}[b]{0.48\textwidth}
    \centering
            \begin{tikzpicture}[scale=2]
      \draw[draw, red, double=red, double distance=2\pgflinewidth] (1, 0) -- (1, 1); 
      \draw[draw, blue, double=blue, double distance=2\pgflinewidth] (0, 0) -- (1, 0); 
      \draw[draw, blue, double=blue, double distance=2\pgflinewidth] (0, 0) -- (0, 1); 
      \draw[draw, blue, double=blue, double distance=2\pgflinewidth] (0, 1) -- (1, 1); 
      \draw (0, 0) rectangle (2, 1) {};
      \node[] at (0.5, 0.5) {$\Omega_L$};
      \node[] at (1.5, 0.5) {$\Omega_N$};

      \draw (1, 0) -- (1, 1); 
      \draw [thick, ->] (1.5, 1.05) -- (1.5, 1.2) node[above] {$\Omega_{\interactionGiven}$};
      \draw [thick, ->] (.95, .15) -- (.75, -.1) node[below] {$\Omega_{\interactionSolve}$};
      \draw[rounded corners, densely dotted] (.9, -.1) rectangle (2.1, 1.1) {};
      \end{tikzpicture}
    \end{subfigure}
    \begin{subfigure}[b]{0.48\textwidth}
    \centering
            \begin{tikzpicture}[scale=1.5]
      \draw[draw, red, double=red, double distance=2\pgflinewidth] (.7, 1) rectangle (1.1, 1.4) {};
      \draw[draw, blue, double=blue, double distance=2\pgflinewidth] (0, 0) rectangle (2, 2) {};

      \draw (0, 0) rectangle (2, 2) {};
      \draw (.7, 1) rectangle (1.1, 1.4) {};

      \node[] at (0.9, 1.2) {$\Omega_N$};
      \node[] at (1.5, 0.5) {$\Omega_L$};
      \draw [thick, ->] (0.9, 1.5) -- (0.9, 1.67) node[above] {$\Omega_{\interactionSolve}$};
      \draw[rounded corners, densely dotted] (.5, .8) rectangle (1.3, 1.6) {};
      \end{tikzpicture}
    \end{subfigure}
    \caption{
      Sketch of geometric configurations of the subdomains: left-right splitting on the \emph{left} and a nonlocal inclusion on the \emph{right}; in \emph{top}: 1D, \emph{bottom}: 2D
      The red/blue highlighted boundaries are $\localBoundSolve, \localBoundGiven$ respectively.
      Note that for simplicity, the domains were chosen such that the overlap \(\Omega_{\local}\cap\Omega_{\nonlocal}\) is empty.
    }
    \label{fig:simple-domain}
\end{figure}
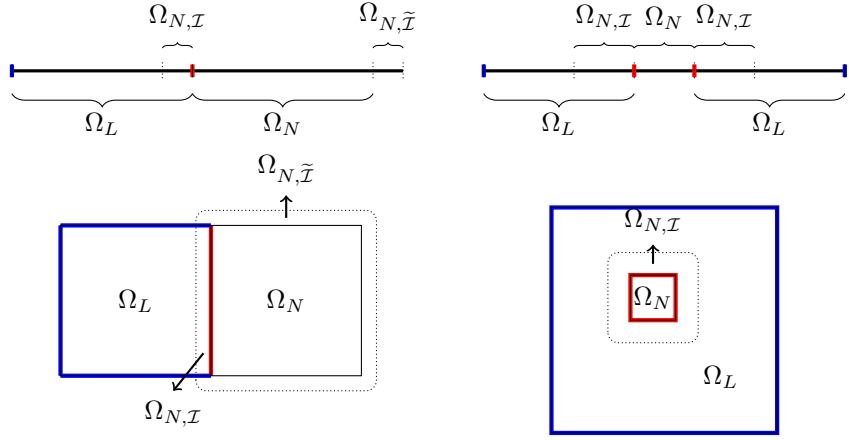

\subsection{Discretization using Finite Element Spaces of Identical Type}
\label{sec:discr-using-finite}
We assume that we have chosen a domain \(\Omega_{\local}\subset\Omega\).
In what follows, we will define a corresponding domain \(\Omega_{\nonlocal}\) with minimal overlap and continuous piecewise linear finite element spaces.

Let \(\mesh\) be a quasi-uniform mesh of $\Omega$ where the intersection of any two distinct elements is a point, an entire edge or the empty set.
We assume that \(\mesh\) contains a submesh \(\mesh[\local]\) that triangulates the domain \(\Omega_{\local}\) with the boundaries of $\Omega_\local$ exactly resolved.
We define the nonlocal mesh \(\mesh[\nonlocal]\) to be
\begin{align*}
  \mesh[\nonlocal] := \left\{K\in\mesh \mid \overline{K}\cap \overline{\Omega\setminus\Omega_{\local}}\neq \emptyset\right\},
\end{align*}
i.e. all elements \(K\) of \(\mesh\) that are outside of $\Omega_\local$ and one layer of elements inside of \(\Omega_{\local}\).
We define \(\Omega_{\nonlocal}\) as the union of all elements in \(\mesh[\nonlocal]\).
In particular, this means that \(\Omega_{\local}\) and \(\Omega_{\nonlocal}\) have an \(\mathcal{O}(h)\) overlap.

We also observe that this construction assures that the vertices of \(\mesh\) are partitioned into three sets: vertices in the interior of \(\Omega_{\local}\), vertices in the interior of \(\Omega_{\nonlocal}\) and vertices on \(\partial\Omega\).
Furthermore, we assume that \(\Omega_{\interactionSolve}\) is resolved exactly by a submesh \(\mesh[\interactionSolve]\) of \(\mesh[\local]\) and that there is a mesh \(\mesh[\interactionGiven]\) for \(\Omega_{\interactionGiven}\) that matches \(\mesh[\nonlocal]\) at the interface between \(\Omega_{\nonlocal}\) and \(\Omega_{\interactionGiven}\) as well as on the interface between \(\Omega_{\interactionSolve}\) and \(\Omega_{\interactionGiven}\).
In other words, we may assume there exists a global mesh on all of $\Omega \cup \Omega_{\interactionGiven}$ such that $\mesh[\local]$, \(\mesh[\nonlocal]\), $\mesh[\interactionSolve]$ and $\mesh[\interactionGiven]$ are submeshes.
We refer the reader to \cref{fig:dofs} for two figures of example meshes.

\begin{remark}\label{rem:domain-definition}
  A particular consequence of the above construction is that \(\Omega_{\nonlocal}\) is mesh dependent.
  For example, if \(\Omega=(-1,1)\) and \(\Omega_{\local}=(-1,0)\), and \(\mesh\) is uniform of size \(h\), then \(\Omega_{\nonlocal}=(-h,1)\).
  Alternatively, one might prefer to partition the nodal degrees of freedom of the \(\mathbb{P}_{1}\) finite element space supported by \(\mesh\) into the disjoint index sets \(\is[\local]\) and \(\is[\nonlocal]\) and define the subdomains as the support of the finite element sub-spaces.
  I.e. \(\Omega_{\local}\) is the support of finite element functions whose degrees of freedom are located at \(x\) with \(x<0\) and \(\Omega_{\nonlocal}\) is the support of functions spanned by basis functions with coordinate \(x\geq0\).
\end{remark}

Let 
\begin{align*}
 V_{h, L} := \{\mathfrak u \in H^1(\Omega_\local) \mid \mathfrak u|_{K} \in \mathbb{P}_1, \forall K \in \mathcal T_{L} \}
\end{align*}
be the standard, $C^0$-continuous FE space on $\mesh[\local]$, and let $V_{h, \local, 0} := \{\mathfrak u \in V_{h, \local} \mid \mathfrak u|_{\partial \Omega_\local} = 0\} $ be the zero Dirichlet boundary condition subspace.
Here and in what follows we use Gothic script \(\mathfrak u\) to distinguish FE functions from generic elements \(u\) of the Sobolev spaces.
The weak-form for \cref{eqn:local-model} on the mesh is as usual: find $\mathfrak u_\local \in V_{h, \local}$ such that for all $\mathfrak v_L \in V_{h,\local,0}$
\begin{align}\label{eqn:weak-local}
    \int_{\Omega_{\local}} \nabla  \mathfrak u_\local \cdot \nabla  \mathfrak v_\local  \, dx = \int_{\Omega_{\local}} f  \mathfrak v_\local \, dx
\end{align}
with essential boundary conditions $\mathfrak u_\local = \mathfrak u_\nonlocal$ on $\localBoundSolve$ and $\mathfrak u_\local = g$ on $\localBoundGiven$.
\(\mathfrak u_{\nonlocal}\) will be defined below.

\begin{figure}
  \centering
   \begin{subfigure}[b]{0.48\textwidth}
       \centering
       \includegraphics[width=\textwidth]{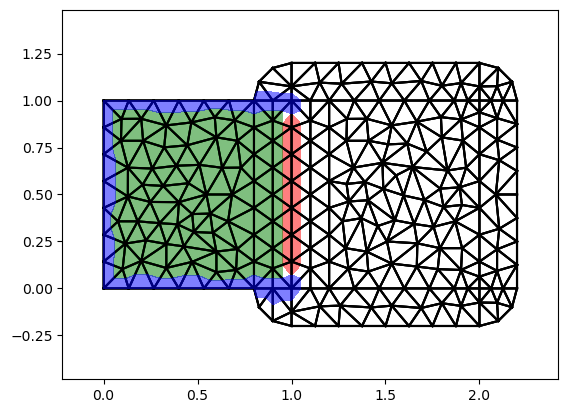}
       \caption{Local}
   \end{subfigure}
   \hfill
   \begin{subfigure}[b]{0.48\textwidth}
       \centering
       \includegraphics[width=\textwidth]{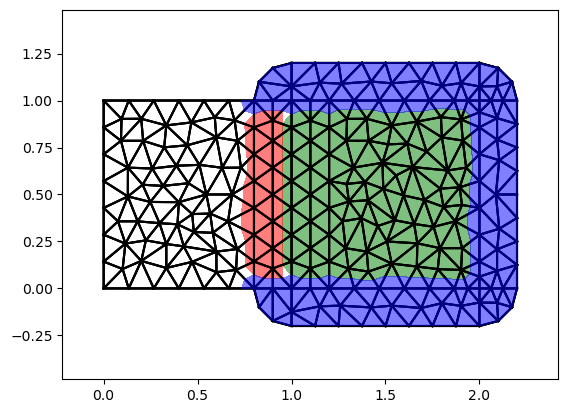}
       \caption{Nonlocal}
   \end{subfigure}\\
      \begin{subfigure}[b]{0.48\textwidth}
       \centering
       \includegraphics[width=\textwidth]{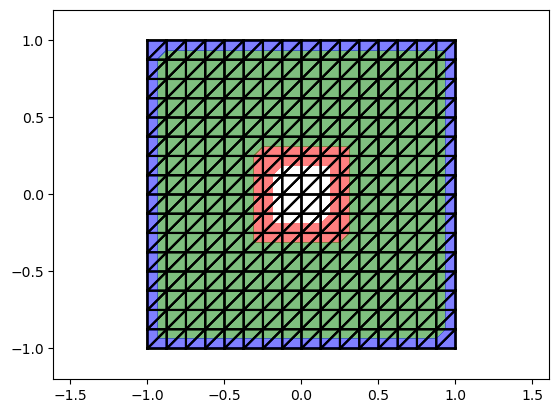}
       \caption{Local}
   \end{subfigure}
   \hfill
   \begin{subfigure}[b]{0.48\textwidth}
       \centering
       \includegraphics[width=\textwidth]{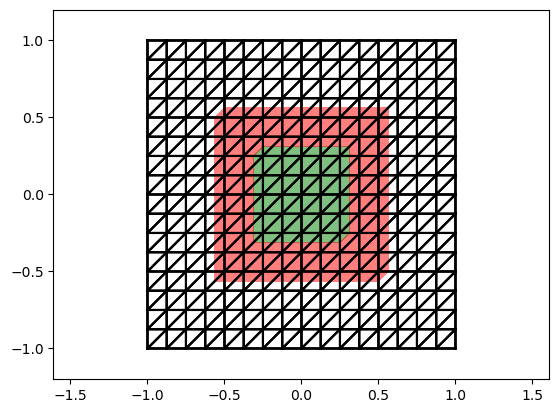}
       \caption{Nonlocal}
   \end{subfigure}\\

   \caption{
     Figure illustrating the degrees of freedom and the corresponding index sets.
     The color green corresponds to $\is[\local]$, $\is[\nonlocal]$, red corresponds to $\is[\localBoundSolve]$, $\is[\interactionSolve]$ and blue corresponds to $\is[\localBoundGiven]$, $\is[\interactionGiven]$ in the local/nonlocal case respectively.
     Note that \(\is[\localBoundSolve]\subset \is[\nonlocal]\) and \(\is[\interactionSolve]\subset \is[\local]\).
     We show two local-nonlocal splittings.
     The top row shows \(\Omega=(0,2)\times(0,1)\) with $\Omega_\local = (0, 1)^{2}$.
     The bottom row shows \(\Omega=(-1,1)^{2}\) with \(\Omega_\local = \Omega\setminus[-.25, .25]^2\).
   }
  \label{fig:dofs}
\end{figure}

For all the degrees of freedom (dofs) of $V_{h, \local}$, there is a natural identification with its coordinate by the usual finite element triple \cite{ciarlet2002finite}.
Let $\is[\local]$ be the index set of degrees of freedom of $\mesh[\local]$ which is strictly in the interior of the domain $\Omega_\local$, 
let $\is[\localBoundSolve]$ correspond to the dofs on $\localBoundSolve$ of $\mesh[\local]$ and similarly for $\is[\localBoundGiven]$.
This means that $\is[\local] + \is[\localBoundSolve] + \is[\localBoundGiven]$ is all the vertices of $\mesh[\local]$. 
See \cref{fig:dofs}a and \cref{fig:dofs}c for an illustration of the index sets.

Using the coefficient vectors \(\vec{u}_{\local}\), \(\vec{u}_{\localBoundSolve}\) and \(\vec{u}_{\localBoundGiven}\) we can write 
\begin{align}\label{eqn:basis-func-local-def}
  \mathfrak u_{\local}=\vec{u}_{\local}\cdot\vec{\Phi}_{\local}+\vec{u}_{\local,\localBoundSolve}\cdot\vec{\Phi}_{\localBoundSolve}+\vec{u}_{\local,\localBoundGiven}\cdot\vec{\Phi}_{\localBoundGiven}.
\end{align}
Here \(\vec{\Phi}_{\local}\) is the vector of basis functions corresponding to the index set \(\is[\local]\).
\(\vec{\Phi}_{\localBoundSolve}\) and \(\vec{\Phi}_{\localBoundGiven}\) are similarly defined.

With the dofs partitioned, \cref{eqn:weak-local} can be written as a blocked linear system
\begin{align*}
  \begin{bmatrix}
        \lapl[\local, \local] & \lapl[\local, \localBoundSolve]  & \lapl[\local, \localBoundGiven] \\
        & \matr{I}_{\localBoundSolve}& \\
        &                            & \matr{I}_{\localBoundGiven}
    \end{bmatrix}
    \begin{bmatrix}
        \vec u_{\local} \\ \vec u_{\local,\localBoundSolve} \\ \vec u_{\local,\localBoundGiven}
    \end{bmatrix}
    =
    \begin{bmatrix}
        \vec f_{\local} \\ \vec u_{\nonlocal,\localBoundSolve} \\ \vec g_{\localBoundGiven}
    \end{bmatrix}
\end{align*}
where the subscripts indicate the interaction between the dofs of the index sets.
The vectors $\vec f_{\omega}$, $\vec g_{\omega}$ indicates the forcing term or essential boundary condition on $\omega$ respectively.

By eliminating the degrees of freedom of the essential boundary condition on \(\localBoundGiven\), we can rewrite the system as
\begin{align}\label{eqn:localmatrix}
    \begin{bmatrix}
        \lapl[\local, \local] & \lapl[\local, \localBoundSolve]  \\
        & \matr{I}_{\localBoundSolve} \\
    \end{bmatrix}
    \begin{bmatrix}
        \vec u_{\local} \\ \vec u_{\local,\localBoundSolve}
    \end{bmatrix}
    =
    \begin{bmatrix}
        \vec f_{\local} - \lapl[\local, \localBoundGiven]\vec g_{\localBoundGiven}\\ \vec u_{\nonlocal,\localBoundSolve}
    \end{bmatrix}.
\end{align}

\begin{remark}\label{rem:discrete-cont}
Henceforth in the manuscript, bold face notation $\vec u$ indicates either the vector and functional form whichever makes sense from context by using the isomorphism \cref{eqn:basis-func-local-def} and the nonlocal equivalent.
\end{remark}

Similarly, for the nonlocal region, we let 
\begin{align}\label{eqn:nonlocal-fem-space}
    V_{h, \nonlocal} &= 
         \{\mathfrak u \in H^1(\overline{\Omega}_\nonlocal \cup \overline{\Omega}_{\interactionGiven} \cup \overline{\Omega}_{\interactionSolve}) \mid \mathfrak u|_{K} \in \mathbb{P}_1, \forall K \in \mathcal T_{\globalNonlocal} \}.
\end{align}
where $\mesh[\globalNonlocal] = \mesh[\nonlocal] \cup \mesh[\interactionSolve] \cup \mesh[\interactionGiven]$.
Let
\begin{align*}
  V_{h, \nonlocal, 0} := \{\mathfrak u \in V_{h, \nonlocal} \mid \mathfrak u|_{\Omega_{\interactionSolve}\cup \Omega_{\interactionGiven}}=0  \}.
\end{align*}
We solve the variational problem corresponding to \cref{eqn:nonlocal-model}:
find $\mathfrak u_\nonlocal\in V_{h,\nonlocal}$ such that for all $\mathfrak v_\nonlocal\in V_{h,\nonlocal,0}$
\begin{align}\label{eqn:weak-nonlocal}
    \frac{1}{2}\iint_{\mathbb{R}^d} {(\mathfrak u_\nonlocal(\vec x) - \mathfrak u_\nonlocal(\vec y))(\mathfrak v_{\nonlocal}(\vec x) - \mathfrak v_{\nonlocal}(\vec y))} \gamma(\vec x, \vec y) \, dx dy= \int_{\Omega_\nonlocal} f  \mathfrak v_\nonlocal \, dx
\end{align}
with Dirichlet boundary conditions $\mathfrak u_\nonlocal = g$ on $\Omega_{\interactionGiven}$ and $\mathfrak u_\nonlocal= \mathfrak u_\local$ on $\Omega_{\interactionSolve}$.

As in the local domain, we identify the dofs with their coordinates.
Let $\is[\nonlocal]$ be the set of dofs of $\mesh[\nonlocal]$ which is within $\Omega_\nonlocal$ including those dofs which lie on the intersection $\partial \Omega_\nonlocal \cap \partial \Omega_\local$.
Let $\is[\interactionSolve]$, $\is[\interactionGiven]$ be the dofs corresponding to $\Omega_{\interactionSolve}$, $\Omega_{\interactionGiven}$ respectively.
Thus, $\is[\nonlocal] + \is[\interactionSolve] + \is[\interactionGiven]$ gives a disjoint partitioning of the dofs of $\mesh[\globalNonlocal]$.
See \cref{fig:dofs}b and \cref{fig:dofs}d for an illustration of the index sets.
As before we can write \(\mathfrak u_{\nonlocal}=\vec{u}_{\nonlocal}\cdot\vec{\Phi}_{\nonlocal}+\vec{u}_{\nonlocal,\interactionSolve}\cdot\vec{\Phi}_{\interactionSolve}+\vec{u}_{\nonlocal,\interactionGiven}\cdot\vec{\Phi}_{\interactionGiven}\).

Using these index sets \Cref{eqn:weak-nonlocal} may be written as a blocked linear equation
\begin{align*}
    \begin{bmatrix}
        \lapl[\nonlocal, \nonlocal] & \lapl[\nonlocal, (\interactionSolve)] & \lapl[\nonlocal, (\interactionGiven)] \\
                                    & \matr{I}_{\interactionSolve}        & \\
                                    &                                     & \matr{I}_{\interactionGiven}  \\
    \end{bmatrix}
  \begin{bmatrix}
        \vec u_{\nonlocal} \\ \vec u_{\nonlocal,(\interactionSolve)} \\ \vec u_{\nonlocal,(\interactionGiven)}
    \end{bmatrix}=
  \begin{bmatrix}
        \vec f_{\nonlocal} \\ \vec u_{\local,(\interactionSolve)} \\ \vec g_{\interactionGiven}
    \end{bmatrix}.
\end{align*}
By eliminating the degrees of freedom of the essential volume condition, we can rewrite the system as
\begin{align}\label{eqn:nonlocalmatrix}
    \begin{bmatrix}
        \lapl[\nonlocal, \nonlocal] & \lapl[\nonlocal, (\interactionSolve)] \\
                                    & \matr{I}_{\interactionSolve}
    \end{bmatrix}
  \begin{bmatrix}
        \vec u_{\nonlocal} \\ \vec u_{\nonlocal,(\interactionSolve)}
    \end{bmatrix}=
  \begin{bmatrix}
        \vec f_{\nonlocal} - \lapl[\nonlocal, (\interactionGiven)]\vec g_{\interactionGiven} \\ \vec u_{\local,(\interactionSolve)} \\
    \end{bmatrix}.
\end{align}
Having detailed the discretization of local and nonlocal problems, we can now describe their coupling to each other, as is already foreshadowed by the choice of notation.

\section{Local-to-Nonlocal Coupling for Weak Forms}
\label{sec:local-nonl-coupl}

Splice methods are commonly used for coupling of local and nonlocal equations in strong form \cite{galvanetto2016effective,silling2015variable,shojaei2016coupled}.
The splicing methods in literature consist of coupling between meshfree methods for nonlocal descriptions and finite elements for the local equations.
In contrast, we present a splice method that combines weak forms for both the local and the nonlocal problem.

\subsection{Matrix Definition and Computational Properties}
Let $n_\local = \abs{\is[\local]}$, $n_\nonlocal = \abs{\is[\nonlocal]}$, $n_\interactionSolve = \abs{\is[\interactionSolve]}$, $n_\localBoundSolve = \abs{\is[\localBoundSolve]}$ denote the total number of degrees of freedom corresponding to $\Omega_L$, $\Omega_N$, $\Omega_{\interactionSolve}$, $\Gamma$ respectively.
Due to the construction of the meshes, $\is[\nonlocal] + \is[\local]$ is a disjoint partition of all the dofs that are inside $\Omega$.
By the same token, all interior dofs that are shared between the two sub-problems can be partitioned into \(\is[\interactionSolve]\) and \(\is[\localBoundSolve]\).
Let $n = n_\local + n_\nonlocal$.
Note that $n$ is the number of dofs which are undetermined given a boundary value problem on the mesh, and hence a global dof $\is$ ordering can be imposed.

Let $\restr[{\local}] : \mathbb{R}^n \to \mathbb{R}^{n_\local}$ and $\restr[{\nonlocal}]: \mathbb{R}^n \to \mathbb{R}^{n_\nonlocal}$ be restriction operators from the global ordering to the local/nonlocal region.
In particular, since we assume each of the $n$ dofs correspond to either local/nonlocal region, we may assume without loss of generality that 
\begin{align*}
    \begin{bmatrix}
        \restr[{\local}] \\
        \restr[\nonlocal]
    \end{bmatrix} = \matr{I}\in\mathbb{R}^{n \times n}.
\end{align*}
Furthermore, let $\restr[{\localBoundSolve}] : \mathbb{R}^n \to \mathbb{R}^{n_{\localBoundSolve}}$ and $\restr[{\interactionSolve}]: \mathbb{R}^n \to \mathbb{R}^{n_\interactionSolve}$ be restriction operators from the global ordering to the local/nonlocal subdomain's interior boundary conditions.
In \cref{fig:dofs}, we see that $\restr[\local]$, $\restr[\nonlocal]$ correspond to the green dofs, while $\restr[\localBoundSolve]$, $\restr[\interactionSolve]$ correspond to the red dofs.

With the above, we can define the splicing LtN coupling: find $\vec u^{S}= \left[ \vec{u}_{\local}; \vec{u}_{\nonlocal}\right] \in \mathbb{R}^n$ such that
\begin{align}
    \matr A^{S} \vec u^{S} = \vec h^{S}
\end{align}
where
\begin{align}\label{eqn:splice-operator}
    \matr A^{S} :=
        \restr[\local]^T \begin{bmatrix}
        \lapl[\local, \local] & \lapl[\local, \localBoundSolve] 
    \end{bmatrix}\begin{bmatrix}
        \restr[\local] \\
        \restr[\localBoundSolve]
    \end{bmatrix} + 
        \restr[\nonlocal]^T \begin{bmatrix}
        \lapl[\nonlocal, \nonlocal] & \lapl[\nonlocal, (\interactionSolve)]
    \end{bmatrix}\begin{bmatrix}
        \restr[\nonlocal] \\
        \restr[\interactionSolve]
    \end{bmatrix}
\end{align}
and 
\begin{align}\label{eqn:splice-rhs}
    \vec h^{S} &:= \restr[\local]^T(\vec f_\local -  \lapl[\local, \localBoundGiven]\vec g_{\localBoundGiven})+
        \restr[\nonlocal]^T (\vec f_\nonlocal - \lapl[\nonlocal, (\interactionGiven)] \vec g_{\interactionGiven}) \\
  &=\vec{f} - \restr[\local]^T\lapl[\local, \localBoundGiven]\vec g_{\localBoundGiven}- \restr[\nonlocal]^T \lapl[\nonlocal, (\interactionGiven)] \vec g_{\interactionGiven}. \nonumber
\end{align}
As usual, the transpose of the restriction operators $\restr[\cdot]^T$ are the corresponding prolongation operator.

\begin{figure}
    \centering
    \includegraphics[width=.5\textwidth]{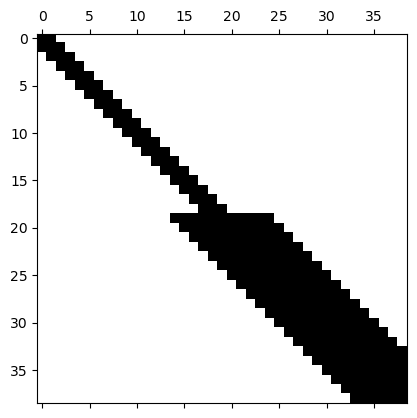}   
    \caption{Figure illustrating the splice matrix for a 1D example. 
    The top half rows correspond to dofs in the local subdomain, resulting in a bandwidth of 3 while the bottom half are dofs on the nonlocal subdomain.}
    \label{fig:matrix-spy}
\end{figure}

The splicing operator $\matr A^{S}$ is non-symmetric.
Each row of $\matr A^{S}$ corresponds to obtaining the \emph{entire} row from either \cref{eqn:localmatrix,eqn:nonlocalmatrix} depending on the location of the dof of that row.
See \cref{fig:matrix-spy} for an illustration of $\matr A^{S}$ in 1D where the first 18 dofs lie on the local domain, while the remaining are in the nonlocal domain.
The local domain has three entries per row, whereas the bandwidth in the nonlocal domain is much wider, commensurate with \(1+2\delta/h\).

Since the coupling method is not intrusive, it is easy to implement. 
In particular, no locally modified kernels or complicated boundary conditions are needed as compared to other coupling methods.

A major advantage of the splicing method is that $\matr A^{S}$ is cheaper to assemble than the nonlocal stiffness matrix on the entire domain.
In particular, only the dofs within $\Omega_N$ need expensive quadrature rules for nonlocal operators, while the remaining dofs can use the standard, integer-order construction.
We will see in time that we incur relatively little error given that $\Omega_\nonlocal$ is chosen appropriately.

On the other hand, an even simpler (but computationally inefficient) way to construct \(\matr A^{S}\) is to compute the matrices \(\matr A^{\local}\) and \(\matr A^{\nonlocal}\) corresponding to the fully local and nonlocal problems where $\Omega_\local = \Omega$ or $\Omega_\nonlocal = \Omega$ respectively and compute the coupled system via 
\begin{align}\label{eqn:spliced-alternate}
  \matr A^{S}=\matr R_{\local}^{T}\matr R_{\local}\matr A^{\local}+\matr R_{\nonlocal}^{T}\matr R_{\nonlocal}\matr A^{\nonlocal}.
\end{align}

Finally, we note that while the conjugate gradient method is no longer a valid solver choice for $\matr A^{S}$ due to lack of symmetry, we suspect that efficient preconditioners can be easily designed for $\matr A^{S}$ and used with Krylov methods such as GMRES or BiCGStab.
We delay this topic for future investigation.

\subsection{Extension to \texorpdfstring{$\mathbb{P}_0$}{P0} Elements}\label{sec:p0-section}
For nonlocal kernels with limited regularity lifting, it may be preferable to choose a discontinuous finite element space for the nonlocal problem.
We now consider a piecewise constant $\mathbb{P}_0$ space for the nonlocal domain, but keep the piecewise linear continuous space for the local discretization.
Thus, instead of \cref{eqn:nonlocal-fem-space} we now have
\begin{align*}
    V_{h, \nonlocal} &= \{\mathfrak u \in L^2(\overline{\Omega}_\nonlocal \cup \overline{\Omega}_{\interactionGiven} \cup \overline{\Omega}_{\interactionSolve}) \mid \mathfrak u|_{K} \in \mathbb{P}_0, \forall K \in \mathcal T_{\mathring  N} \}
\end{align*}
and the dofs of $V_{h,N}$ can now be identified with the centroids of $\mesh[\globalNonlocal]$.

Instead of first choosing \(\Omega_{\local}\) and then constructing a domain \(\Omega_{\nonlocal}\) with minimal overlap, we now choose \(\Omega_{\nonlocal}\) first and then define \(\Omega_{\local}\).
We require that the dof coordinates between local and nonlocal domains (including boundaries) match wherever they overlap, meaning a different local domain \(\Omega_{\local}\) and mesh \(\mesh[\local]\) is needed compared to the previous \(\mathbb{P}_{1}\)--\(\mathbb{P}_{1}\) coupling.

For \(k=0,1\) and a mesh \(\mathcal{U}\) let \(\mathcal{D}_{\mathcal{U}}^{\mathbb{P}_{k}}\) be the set of dof coordinates of the \(\mathbb{P}_{k}\) finite element space defined on \(\mathcal{U}\).
The construction of domains and meshes is as follows:
\begin{enumerate}
\item Choose \(\Omega_{\nonlocal}\) (and hence the corresponding interaction region $\Omega_{\interactionGiven}, \Omega_{\interactionSolve}$).
\item Construct a mesh \(\mesh\) on \(\Omega\cup\Omega_{\interactionGiven}\) (e.g. even on the region outside of $\Omega_\nonlocal$) with submeshes \(\mesh[\nonlocal]\), $\mesh[\interactionSolve]$ and $\mesh[\interactionGiven]$.
\item Construct mesh \(\mesh[\local]\) having vertices
  \begin{align*}
    &\left\{\vec{x}\in \mathcal{D}_{\mesh}^{\mathbb{P}_{0}} \mid \vec{x}\in K\in \mesh\setminus\mesh[\interactionGiven] \text{ and } \right. \\
    &\qquad\left.\overline{K} \cap \overline{(\Omega\setminus\Omega_{\nonlocal})} \text{ is a \(d\) or \(d-1\) dimensional sub-simplex of } \mesh\right\} \\
    &\cup\left\{\vec{x}\in\mathcal{D}_{\mesh}^{\mathbb{P}_{1}} \mid \vec{x}\in \partial\Omega \setminus \Omega_{\nonlocal}\right\}.
  \end{align*}
\item
  Take \(\Omega_{\local}\) to be the union of elements of \(\mesh[\local]\).
\end{enumerate}

In step 3, the first set of vertices corresponds to the \(\mathbb{P}_{0}\) dof coordinates in \(\Omega\) that are on elements which are part of or share a \(d-1\) dimensional sub-simplex (a vertex for \(d=1\) and an edge for \(d=2\)) with \(\Omega\setminus\Omega_{\nonlocal}\).
The second set is simply the vertices on the boundary of the domain. 
In practice, step 3 can be implemented using any mesh generation algorithm that allows one to pre-define the set of vertices.
We refer the reader to \cref{fig:p0-p1-coupling-meshes} for example meshes in both 1D and 2D.

\begin{figure}
       \centering
     \begin{subfigure}[b]{0.48\textwidth}
         \centering
         \includegraphics[width=\textwidth]{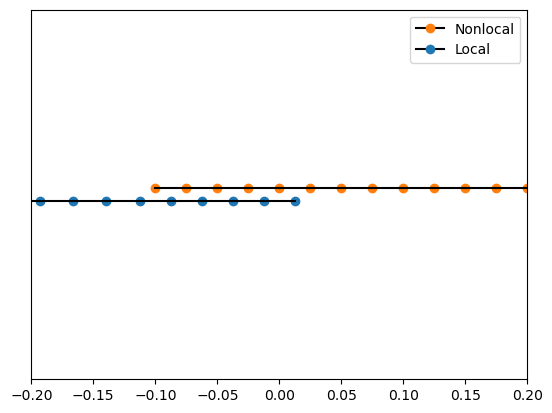}
         \caption{1D: the blue nodes correspond to dofs of the $\mathbb{P}_1$ mesh while the dofs of the $\mathbb{P}_0$ mesh lies at the midpoints between two orange nodes.}
     \end{subfigure}
     \hfill
     \begin{subfigure}[b]{0.48\textwidth}
         \centering
         \includegraphics[width=\textwidth]{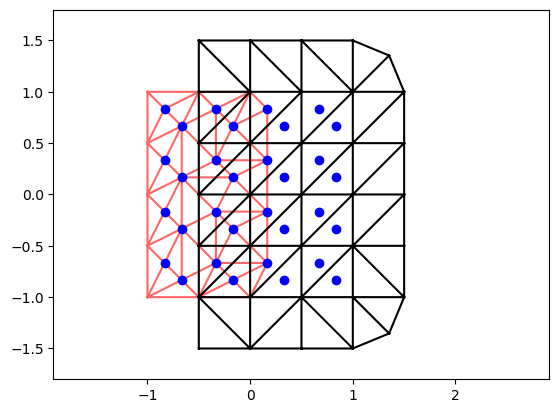}
         \caption{2D: the blue nodes correspond to the dofs of the two meshes. 
         The red edges are the edges of the $\mathbb{P}_1$ mesh, and the black edges correspond to the $\mathbb{P}_0$ mesh.}
                  \label{eqn:good-fig-p0p1}
     \end{subfigure}
      \begin{subfigure}[b]{0.48\textwidth}
         \centering
         \includegraphics[width=\textwidth]{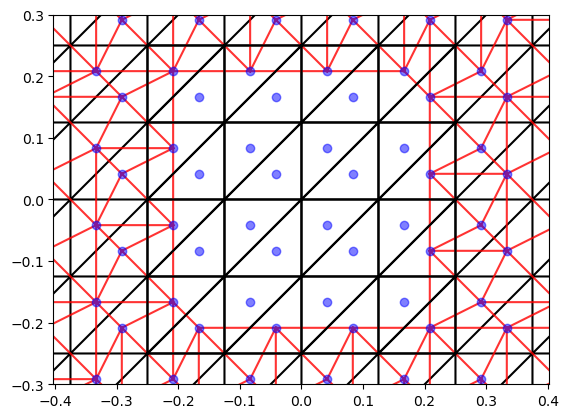}
         \caption{2D: the blue nodes correspond to the dofs of the two meshes. 
         The red edges are the edges of the $\mathbb{P}_1$ mesh, and the black edges correspond to the $\mathbb{P}_0$ mesh.}
     \end{subfigure}

     \caption{
       Example meshes for the non-matching meshes coupling.
       See \cref{rem:domain-nonmatching} regarding the notation of the domains. 
       In the 1D case, the domains are $\Omega_\local = (-1, 0), \Omega_\nonlocal = [0, 1)$ and for 2D they are $\Omega_\local = (-1, 0) \times (-1, 1), \Omega_N = [0, 1) \times (-1, 1)$ and $ \Omega_\local = (-1, 1)^2 \setminus [-.25, .25]^2, \Omega_N = (-.25 , .25 )^2$.
     }
     \label{fig:p0-p1-coupling-meshes}
\end{figure}

As a result of having constructed the local mesh in a manner that guarantees that local and nonlocal dofs are collocated in the interaction region, we can repeat the construction of the splice matrix as in the case of identical discretizations.

As an example for concreteness, we will refer to \cref{eqn:good-fig-p0p1} as a reference figure. 
The same quantities can be defined as before: let $n = \abs{\is[\nonlocal]} + \abs{\is[\local]}$ be the total number of unique dofs in the interior of $\Omega$.
In the case of \cref{eqn:good-fig-p0p1}, $\is[\local]$ corresponds to the 16 dofs strictly to the left of $x = 0$ while $\is[\nonlocal]$ are the 16 dofs strictly to the right, making $n = 32$.
$\is[\interactionSolve]$ are the 8 dofs to the left of $x = 0$ for the nonlocal mesh and $\is[\localBoundSolve]$ are the four dofs on the right of $x = 0$ corresponding to the $\mathbb{P}_1$ mesh.
While we do not have a global mesh that contains \(\mesh[\local]\) and \(\mesh[\nonlocal]\), the key observation is that we can impose a global ordering of the dofs, meaning that restriction matrices $\restr[\nonlocal]$, $\restr[\local]$, $\restr[\interactionSolve]$, $\restr[\localBoundSolve]$ can be introduced again.
Thus, the \emph{exact same operator} \cref{eqn:splice-operator} can be used for the splice coupling.

\begin{remark}\label{rem:domain-nonmatching}
  Similar to the domains defined in the matching case, it is not easy to define the exact domains. 
  Furthermore, since each interior dof can be classified as either a local $\is[\local]$ or a nonlocal dof $\is[\nonlocal]$, we employ the same strategy of labeling the dofs again to define $\Omega_\local, \Omega_\nonlocal$ in this case also. 
\end{remark}

\section{Properties of Splice LtN Coupling}
\label{sec:prop-splice-ltn}

The direct analysis of the above splice coupling matrix is difficult to perform, mainly due to the lack of symmetry in the resulting linear system.
However, we can show that there exists a unique solution to the splice operator in the $\mathbb{P}_1 - \mathbb{P}_1$ coupling case under some additional assumptions, by interpreting the method as an optimization-based LtN coupling similar to the one detailed in \cite{d2016coupling,DEliaBochev2021_FormulationAnalysisComputationOptimization}. 
A rigorous proof for the $\mathbb{P}_0 - \mathbb{P}_1$ coupling could be obtained in similar fashion.

\subsection{Continuous Optimization-Based Coupling}
Let \(\Omega_{b}:=\left(\Omega_{\nonlocal}\cup\Omega_{\interactionSolve}\cup\Omega_{\interactionGiven}\right)\cap \Omega_{L}\) be the overlap between nonlocal and local domains.
For the case of the splice method described above, several constraints were enforced with regards to the local and nonlocal domains, meshes and discretization.
These restrictions can be interpreted as a special case of optimization-based coupling which we first state in the continuous case. 

Let $\theta_{\local} \in H^{1/2}_{0}(\localBoundSolve)$, and $\theta_{\nonlocal} \in \{u \in \operatorname{trace}_{\Omega_\interactionSolve}V(\Omega_{\nonlocal}\cup\Omega_{\interactionSolve}\cup\Omega_{\interactionGiven} ) \mid u|_{\Omega_\interactionGiven} = 0 \}$.
The continuous optimization problem is defined as
\begin{align}\label{eqn:cont-opt}
    \min_{\theta_\local,\theta_\nonlocal} \mathcal{J}(\theta_\local,\theta_\nonlocal) = \frac{1}{2} \int_{\Omega_b} \left(u_\nonlocal- u_\local\right)^2 \, dx
\end{align}
subject to
\begin{align}
    \left\{
    \begin{aligned}
      -\mathcal{L}_\nonlocal u_{\nonlocal}(\vec x) &= f(\vec x), && \vec x \in \Omega_{\nonlocal}, \\
      u_\nonlocal(\vec x) &= \theta_\nonlocal(\vec x), && \vec x \in \Omega_{\interactionSolve}, \\
      u_\nonlocal(\vec x) &= 0,  && \vec x \in\Omega_{\interactionGiven},
    \end{aligned}
    \right.
    \qquad
      \left\{
\begin{aligned}
    -\mathcal{L}_\local  u_\local(\vec x) &= f(\vec x), && \vec x \in \Omega_{\local}, \\
     u_\local(\vec x) &=  \theta_\local(\vec x), && \vec x \in \localBoundSolve, \\
     u_\local(\vec x) &= 0, && \vec x \in \localBoundGiven.
\end{aligned}
\right.
  \label{eqn:opt-cont-sub}
\end{align}
The trace spaces for $\theta_\local$, $\theta_\nonlocal$ are chosen such that the \cref{eqn:opt-cont-sub} are well-defined.

The two subdomain problems match \eqref{eqn:nonlocal-model} and \eqref{eqn:local-model} with the interior boundary and volume conditions replaced with control variables.
In essence, \cref{eqn:cont-opt} is minimizing the difference of local and nonlocal solutions in the region where both are defined.
We will show that if the local and nonlocal solution are identical in \(\Omega_{b}\), we recover the solution of the splice approach after discretization.

The main appeal of optimization-based LtN coupling is its flexibility.
Since the only interaction between the nonlocal and local domains is through an optimization problem, the method is non-intrusive and only requires one to compute the integral \cref{eqn:cont-opt}.
However, this flexibility comes at a computational cost as the optimization problem, in general, must be solved through an iterative minimization, where each iteration involves solving the two sub-problems of \cref{eqn:opt-cont-sub}.
This is quite undesirable, even more so in time-dependent problems.

Before proceeding, we wish to point out some constraints that were imposed in \cite{d2016coupling,DEliaBochev2021_FormulationAnalysisComputationOptimization} on the subdomains: in order to prove well-posedness of the optimization problem the domains had to be such that the local subdomain $\Omega_\local$ \emph{must} overlap the prescribed nonlocal boundary condition $\Omega_\interactionGiven$ (see \cite[Fig 2.]{DEliaBochev2021_FormulationAnalysisComputationOptimization}).
This means that 1D scenarios such as the left-right/inclusion scenario are not covered nor is the inclusion scenario in higher dimensions.
Nevertheless, many of the techniques from \cite{DEliaBochev2021_FormulationAnalysisComputationOptimization,d2016coupling} can be generalized with slight modifications to be applicable in our discrete case below.
While we focus on splice coupling, we conjecture that our work allows to analyze the optimization-based coupling approach in a more general context as well.

\subsection{Discretized Optimization-Based Coupling}
The discrete version of the optimization problem \eqref{eqn:cont-opt} can be obtained in the same manner as in \Cref{sec:discr-using-finite}.
Let 
\begin{align}
  \Theta_\local := \{v \in \operatorname{trace}_{\Omega_{\localBoundSolve}} \mathfrak w  \mid \mathfrak w \in V_{h, \local},~  \mathfrak w|_{\Omega_{\localBoundGiven}} = 0 \}
\end{align}
the space of discrete functions on the set of dofs in $\is[\localBoundSolve]$ with zero extension.
Similarly,
\begin{align*}
  \Theta_\nonlocal := \{v \in \operatorname{trace}_{\Omega_{\interactionSolve}} \mathfrak w  \mid \mathfrak w \in V_{h, \nonlocal},~  \mathfrak w|_{\Omega_{\interactionGiven}} = 0 \}
\end{align*}
be the space of discrete functions whose dofs are in $\is[\interactionSolve]$.
We note that $\Theta_\local$, $\Theta_\nonlocal$ are simply discrete function spaces on $\Gamma$, \(\Omega_\interactionSolve\) respectively from above and would correspond to the green area and blue line of \cref{fig:cont-case-domains}.

Let us define the discretized versions of \cref{eqn:opt-cont-sub} using finite elements.
For all $\vec \theta_L \in \Theta_L$, we define $\vec u_L(\vec \theta_L) \in V_{h, L}$ to be the solution to
\begin{align}\label{eqn:ul-opt}
     \lapl[\local, \local] \vec u_\local(\vec \theta_\local) = \vec f_\local -  \lapl[\local, \localBoundSolve] \vec \theta_\local 
\end{align}
with Dirichlet boundary conditions $ \vec u_L(\vec \theta_\local) = \vec \theta_L$ on $\is[\localBoundSolve]$. 
Similarly, for all $\vec \theta_N \in \Theta_N$, let $\vec u_N(\vec \theta_N) \in V_{h, N}$ be the solution to
\begin{align}\label{eqn:un-opt}
     \lapl[\nonlocal, \nonlocal] \vec u_N(\vec \theta_N) = \vec f_\nonlocal - \lapl[\nonlocal, (\interactionSolve)] \vec \theta_N 
\end{align}
with boundary conditions $ \vec u_\nonlocal(\vec \theta_\nonlocal) = \vec \theta_\nonlocal$ on $\is[\interactionSolve]$. 
As shorthand, we will let $\vec u_L(\vec \theta_L) := \vec u_L$, $\vec u_N(\vec \theta_N) := \vec u_N$ for simplicity in the rest of the section.

The discrete optimization problem is defined as 
\begin{align}\label{eqn:discrete-opt}
  \min_{\vec \theta_L \in \Theta_L, \vec \theta_N \in \Theta_N} \mathcal{J}(\vec\theta_L, \vec\theta_N) &:= 
  \frac{1}{2} \int_{\Omega_b} \left(\vec u_\nonlocal- \vec u_\local\right)^2 \, dx \\
  \text{subject to \eqref{eqn:ul-opt}, \eqref{eqn:un-opt}} \nonumber
\end{align}
which seeks controls that minimizes the discrepancy between the two solutions wherever they overlap (i.e. $\Omega_b$ from above) in a discrete manner.
\begin{remark}
  The mass matrix corresponding to the integral in the objective function of \cref{eqn:discrete-opt} is spectrally equivalent to a diagonal matrix meaning a simple sum over the dofs can suffice in practice.
\end{remark}

In order to show that \cref{eqn:discrete-opt} is well-defined, several assumptions are needed on the nonlocal kernel.
However, these technical assumptions are satisfied by the kernels discussed in prior the introduction. 
For the exact statements, we refer the reader to the proof of \cref{thm:main-theorem} in the appendix. 
With the additional assumptions, we can show that there exists a unique solution that achieves a minimum of zero and that the solution is equivalent to the splice method:
\begin{theorem}\label{thm:main-theorem}
  The discrete optimization problem \eqref{eqn:discrete-opt} has a unique minimizer $(\vec \theta_\local^*, \vec \theta_\nonlocal^*)$. It satisfies $\mathcal J(\vec \theta_\local^*, \vec \theta_\nonlocal^*) = 0$, and the corresponding solution \((\vec{u}_{\local}(\vec{\theta}_{\local}^{*}),\vec{u}_{\nonlocal}(\vec{\theta}_{\nonlocal}^{*}))\) is the unique solution of the splice equations \eqref{eqn:splice-operator}.
\end{theorem}
The proof and its necessary lemmas are all technical, and delayed until the appendix. 
The above theorem means that the splice approach is equivalent to an optimization coupling method with a $\mathcal O(h)$ overlap between \(\Omega_{\local}\) and \(\Omega_{\nonlocal}\).
This implies that, besides inheriting existence/uniqueness properties, a host of other useful results can be used from \cite{DEliaBochev2021_FormulationAnalysisComputationOptimization}.
In particular, rigorous bounds on the error obtained from substituting a local operator for the nonlocal operator are available \cite[Theorem 5.2]{DEliaBochev2021_FormulationAnalysisComputationOptimization}, and that the method will pass the patch tests, discussed in more detail below.
Furthermore, the method is naturally asymptotically compatible provided that \(\delta\) is kept larger than \(h\).

\begin{figure}
    \centering
      \begin{subfigure}[b]{0.48\textwidth}
         \centering
         \includegraphics[width=\textwidth]{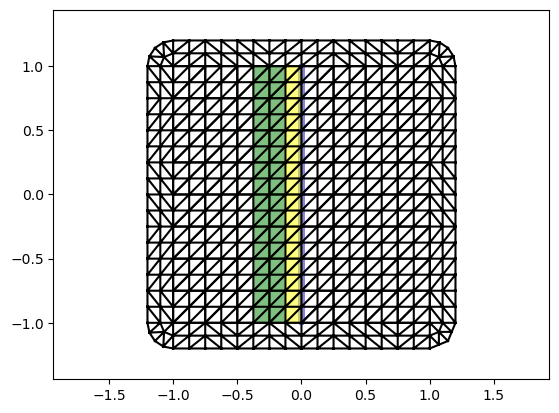}
         \caption{Matching Meshes}
         \label{fig:matching-cont}
     \end{subfigure}
      \begin{subfigure}[b]{0.48\textwidth}
         \centering
         \includegraphics[width=\textwidth]{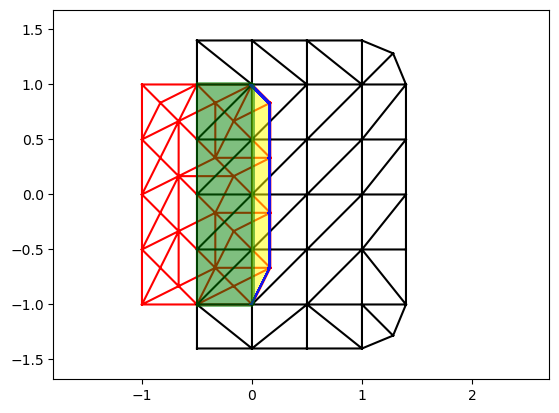}
         \caption{Non-matching Meshes.}
     \end{subfigure}
     \caption{
       The domains of the optimization method.
       \emph{Left:} matching \(\mathbb{P}_{1}-\mathbb{P}_{1}\) discretizations.
       \emph{Right}: non-matching \(\mathbb{P}_{1}-\mathbb{P}_{0}\) discretizations.
       The blue color edges corresponds to the support of $\Theta_{\local}$, the green area corresponds to the support of $\Theta_{\nonlocal}$.
       Green and yellow area together constitute $\Omega_b$.
     }
    \label{fig:cont-case-domains}
\end{figure}

\subsection{Patch Tests}
One fundamental property desired from LtN coupling is the ability to pass so-called \emph{patch tests} otherwise known as consistency tests.
While this property is also inherited from the optimization method, it is both illustrative and straightforward to show directly. 

Assume that there is a pair of functions \(u^{*}, f^{*}\) such that that the pair solves both the continuous fully nonlocal equation \eqref{eqn:nonlocal-model-full} as well as the continuous fully local equation
\begin{equation}\label{eqn:local-model-full}
  \left\{
    \begin{aligned}
    -\mathcal{L}_\local u^{*}(\vec x) &= f^{*}(\vec x)  && \vec{x} \in \Omega, \\
      u^{*}(\vec x) &= g(\vec x) && \vec{x} \in \partial\Omega,
    \end{aligned}
    \right.
\end{equation}
with the appropriate Dirichlet boundary or volume conditions.
Since the solutions to the continuous problems are identical, the solutions of the corresponding discrete problems will also match, at least up to discretization error.
We also want the coupled problem to recover the same solution (up to discretization error) for \emph{arbitrary} domain splittings.
For the diffusion operators that we consider, polynomial solutions up to degree 3 and their corresponding forcing solve both fully local and fully nonlocal problems in the continuous case.

Let $\vec u$, $\vec f$ be the vector corresponding to $u^*$, $f^*$ on a mesh $\mesh$ of some domain $\Omega$.
Recall that $\matr {A}^\local$, $\matr {A}^\nonlocal$ are the Laplacians on all of $\Omega$ meaning that
\begin{align*}
  \matr A^\local \vec u &= \vec f - \matr A^\local_{I, \localBoundGiven} \vec g_{\localBoundGiven}
  &\text{and}&&
  \matr A^\nonlocal \vec u &= \vec f - \matr A^\nonlocal_{I, (\interactionGiven)} \vec g_{\interactionGiven}
\end{align*}
if we neglect the difference in discretization errors.
The discetization errors can be zero since low degree polynomials might be in the finite element space.

Using the alternate form of the splice formulation \cref{eqn:spliced-alternate}, we easily see that 
\begin{align*}
   \matr A^S \vec u &= \restr[\local]^T \restr[\local] \matr A^L \vec u + \restr[\nonlocal]^T \restr[\nonlocal] \matr A^N \vec u \\
   &= \restr[\local]^T \restr[\local] \left(\vec f - \matr A^\local_{I, \localBoundGiven} \vec g_{\localBoundGiven}\right)+ \restr[\nonlocal]^T \restr[\nonlocal] \left(\vec f - \matr A^\nonlocal_{I, (\interactionGiven)} \vec g_{\interactionGiven} \right)  \\
   &= \vec f -  \restr[\local]^T \restr[\local]\matr A^\local_{I, \localBoundGiven} \vec g_{\localBoundGiven} -  \restr[\nonlocal]^T \restr[\nonlocal] \matr A^\nonlocal_{I, (\interactionGiven)} \vec g_{\interactionGiven} = \vec{h}^{S}.
 \end{align*} 
The terms involving the boundary conditions $\vec g_{\localBoundGiven}$, $\vec g_{\interactionGiven}$ correspond exactly to the treatment of boundary conditions in the splice method \cref{eqn:splice-rhs}.
Thus the splice method passes the patch test.

\section{Numerical Results}
\label{sec:numerical-results}

We now present numerical results in \(d=1\) and \(d=2\) dimensions of applying the splice LtN method.
We show patch tests, comparisons against optimization-based coupling and typical applications settings with localized nonlocal behavior for static and time-dependent problems.
Furthermore, we also provide comparison to the optimization-based coupling to show that indeed the optimization objective function will be zero. 
We also showcase the ease of implementation by considering a fractional heat equation with time-dependent forcing terms.

All the examples are implemented using PyNucleus \cite{Glusa2021_PyNucleus}.
PyNucleus is a finite element code that implements efficient assembly and solution of nonlocal and local equations.

\subsection{Patch Tests}

\paragraph{1D Patch Tests}
Let us verify that the splice coupling indeed passes patch tests, and corresponds to the optimization problem with 0 as its objective minimum.
Let $\Omega = (-1,1)$, and consider a fractional kernel \eqref{eq:fracKernel} with horizon $\delta = 0.1$ and fractional order $s = 0.75$.

We choose $\Omega_\local = (-1, 0)$ and use a mesh with mesh size size $h = 0.05$.
This corresponds to a left-right splitting of the domain.

We consider the following two settings for patch tests:
\begin{itemize}
\item forcing \(f=0\) and boundary conditions $u(-1) = -1$ and \(u(x)=x\) for \(x\in[1,1+\delta)\) with analytic solution \(u(x)=x\), and
\item forcing \(f=-2\) and boundary conditions $u(-1) = 1$ and \(u(x)=x^{2}\) for \(x\in[1,1+\delta)\) with analytic solution \(u(x)=x^{2}\), and
\end{itemize}
The solutions resulting from the splice LtN coupling using \(\mathbb{P}_{1}-\mathbb{P}_{1}\) discretization are plotted in \cref{fig:1d-patch-splice}.
The errors are effectively zero in both cases. 

We employ the same problem configurations and solve the optimization problem \eqref{eqn:discrete-opt}.
We use the BFGS minimizer from the \texttt{scipy} library with an initial guess of 0 for the undetermined coefficients.
In \cref{tab:1d-opt-stats}, we show the value of the objective function \cref{eqn:discrete-opt} and the number of iterations arising from BFGS.
we observe that the objective function has essentially a value of zero at the minimum, in accordance with \Cref{thm:main-theorem}.
Note that since the objective function in \cref{eqn:discrete-opt} is quadratic, the pointwise mismatch between local and nonlocal solution is roughly of order $10^{-6}$.
This is confirmed in \cref{fig:1d-patch-error} where we plot the difference between the solution of the splice method and the solution of the optimization problem.
The difference is within the optimization tolerance.

\begin{table}[ht]
    \centering
    \begin{tabular}{|c|cc|}
    \hline
         &  Objective Function Value &BFGS Iterations \\
         \hline
        Linear & 3.833e-13& 7\\ 
        Quadratic &  2.378e-13 & 6\\ 
        \hline
    \end{tabular}
    \caption{Statistics of the optimization procedure on the patch tests for 1D, \(\mathbb{P}_{1}-\mathbb{P}_{1}\) discretizations. 
    }
    \label{tab:1d-opt-stats}
\end{table}

\begin{figure}
     \centering
     \begin{subfigure}[b]{0.48\textwidth}
         \centering
         \includegraphics[width=\textwidth]{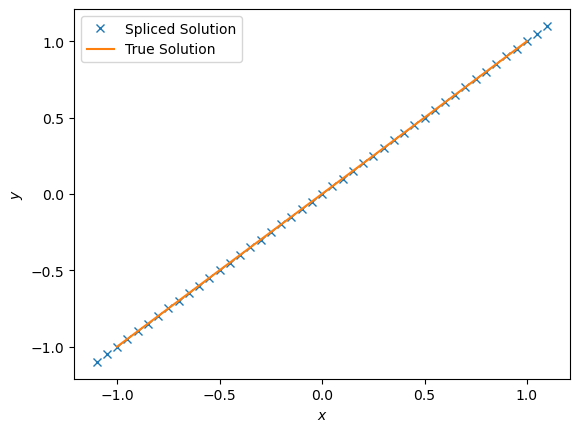}
         \caption{Linear Patch Test Solution}
     \end{subfigure}
     \begin{subfigure}[b]{0.48\textwidth}
         \centering
         \includegraphics[width=\textwidth]{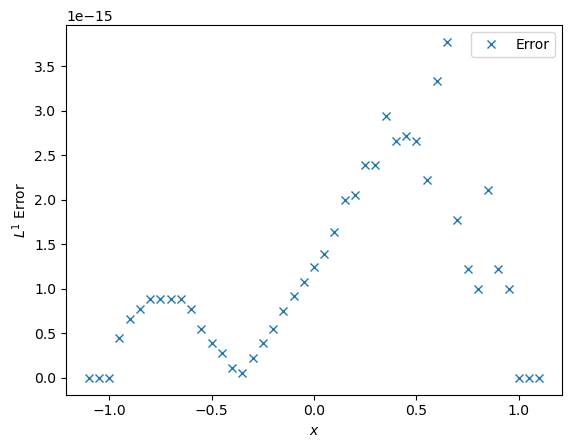}
         \caption{Linear Patch Test Error}
     \end{subfigure} \\
        \begin{subfigure}[b]{0.48\textwidth}
         \centering
         \includegraphics[width=\textwidth]{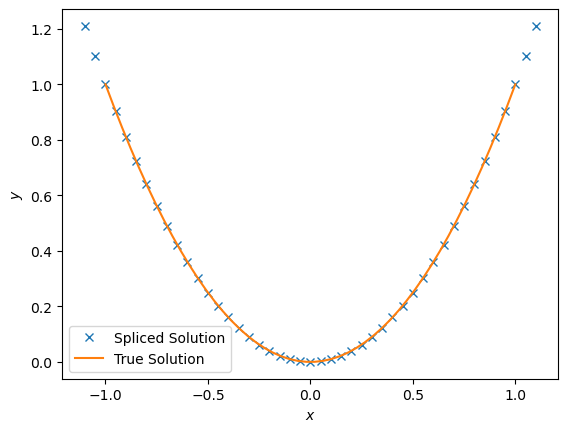}
         \caption{Quadratic Patch Test Solution}
     \end{subfigure}
     \begin{subfigure}[b]{0.48\textwidth}
         \centering
         \includegraphics[width=\textwidth]{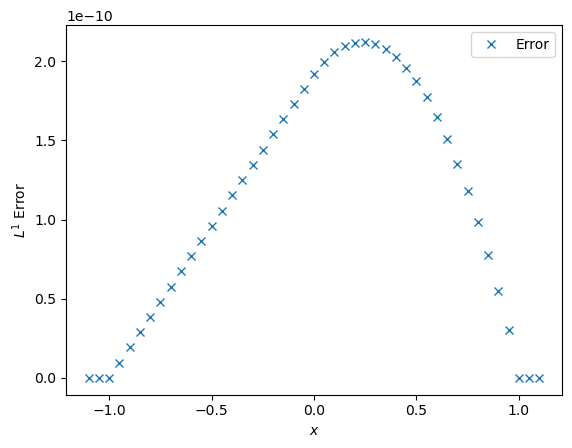}
         \caption{Quadratic Patch Test Error}
     \end{subfigure} \\  
    \caption{Patch tests for the splice method and \(\mathbb{P}_{1}-\mathbb{P}_{1}\) discretization. We recover the analytic solution.}
    \label{fig:1d-patch-splice}
\end{figure}

\begin{figure}
    \centering
     \centering
     \begin{subfigure}[b]{0.48\textwidth}
         \centering
         \includegraphics[width=\textwidth]{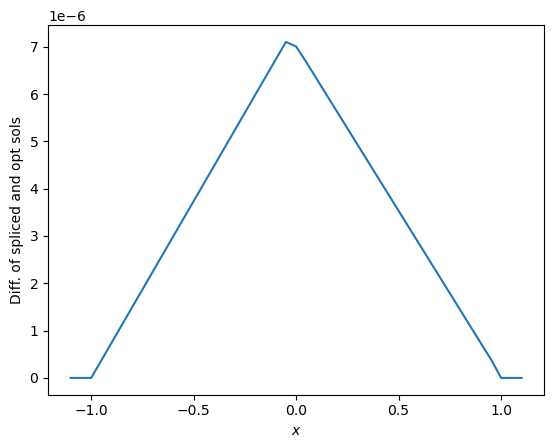}
         \caption{Linear}
     \end{subfigure}
     \begin{subfigure}[b]{0.48\textwidth}
         \centering
         \includegraphics[width=\textwidth]{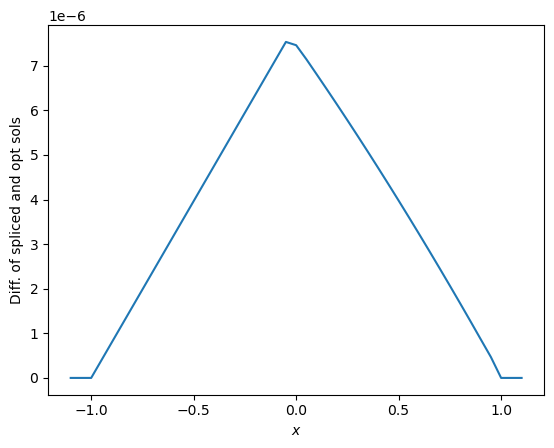}
         \caption{Quadratic}
     \end{subfigure} 
     \caption{
       Difference between solutions obtained using the splice and optimization methods for \(\mathbb{P}_{1}-\mathbb{P}_{1}\) discretization.
       Since the final objective function value is around $10^{-12}$, we expect the pointwise error to be roughly $10^{-6}$ which is what we observe.
       The two differences for the linear and quadratic cases are very similar. 
     }
    \label{fig:1d-patch-error}
\end{figure}

We also perform the patch test for the non-matching $\mathbb{P}_0 -\mathbb{P}_1$ case in one dimension for the same domains, but with different meshes that have been constructed as outlined in \cref{sec:p0-section}.
We also use a fractional kernel of order $s = 0.25$; the horizon remains the same as above.
In \cref{fig:1d-p0-p1-quadratic-patch}, we plot the solutions to the purely nonlocal problem using $\mathbb{P}_0$ elements and the splice LtN coupling.
We see clearly that the splice method again passes the quadratic patch test, as the error in the coupling solution is within the discretization error observed for the fully nonlocal case.

\begin{figure}
    \centering
     \centering
     \begin{subfigure}[b]{0.48\textwidth}
         \centering
         \includegraphics[width=\textwidth]{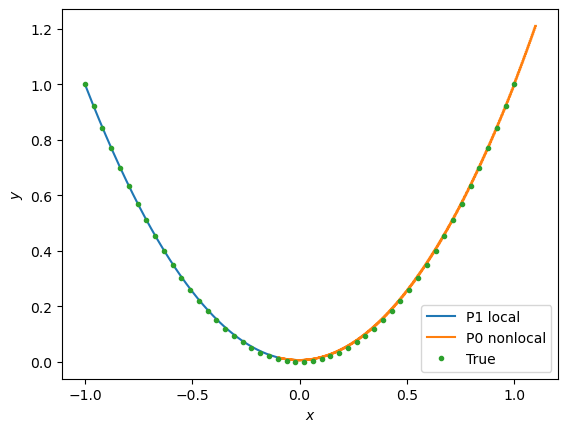}
         \caption{Splice method solution}
     \end{subfigure}
     \begin{subfigure}[b]{0.48\textwidth}
         \centering
         \includegraphics[width=\textwidth]{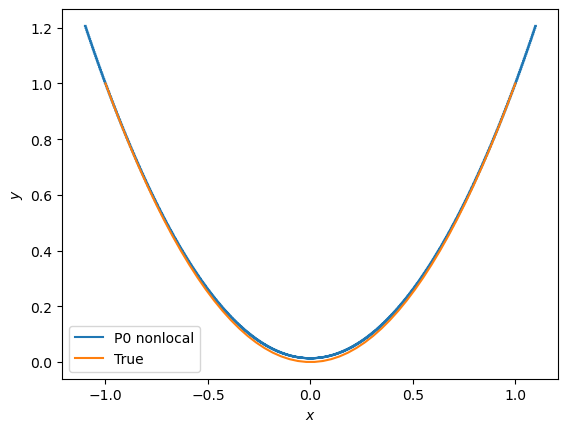}
         \caption{Fully nonlocal solution}
     \end{subfigure} \\
      \begin{subfigure}[b]{0.48\textwidth}
         \centering
         \includegraphics[width=\textwidth]{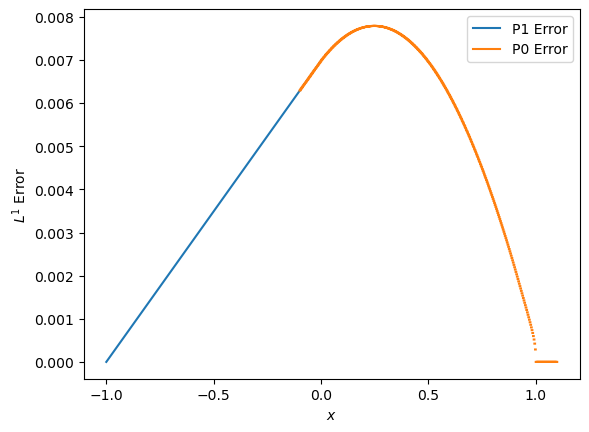}
         \caption{Error for splice method}
     \end{subfigure}
     \begin{subfigure}[b]{0.48\textwidth}
         \centering
         \includegraphics[width=\textwidth]{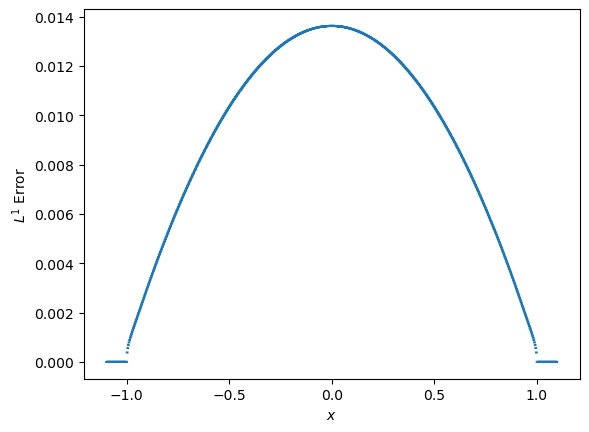}
         \caption{Error for fully nonlocal solution}
     \end{subfigure} \\
    \caption{Plot of the solution and error for the quadratic test for the splice method using a $\mathbb{P}_0-\mathbb{P}_1$ discretization and the fully nonlocal problem using a \(\mathbb{P}_{0}\) discretization.
    }
    \label{fig:1d-p0-p1-quadratic-patch}
\end{figure}

\paragraph{2D Patch Tests}
We perform 2D patch tests in the linear and quadratic cases on the domain $\Omega = (-1, 1)^2$.
We use a fractional kernel with $s = 0.25$ and horizon $\delta = 0.2$.
We consider two different configurations for the splitting into nonlocal/local domain:
\begin{itemize}
\item Left-right: $\Omega_L := (-1, 0) \times (-1, 1)$ splitting the computational domain vertically.
\item Inclusion: $\Omega_L := \Omega  \setminus [-.25, .25]^2$ and the nonlocal domain is contained within the local domain.
\end{itemize}

Let the analytic solution be $u(x, y) = 2(x - 1)^2 - y + 2$ and pick boundary conditions and forcing term accordingly.
A quasi-uniform mesh of $h\approx 0.08$ is used.
The error resulting from the splice coupling in the matching mesh case is shown in \cref{fig:patch-2d} where it can be seen that the errors are in fact smaller than the discretization error incurred in the fully nonlocal case.
It can be observed that the error of the coupling method is concentrated within the nonlocal subdomain.
This is due to quadrature error from the approximation of \(\ell_{2}\) interaction balls in dimension \(d\geq2\).
\begin{figure}
    \centering
    \begin{subfigure}[b]{0.32\textwidth}
        \includegraphics[width=\textwidth]{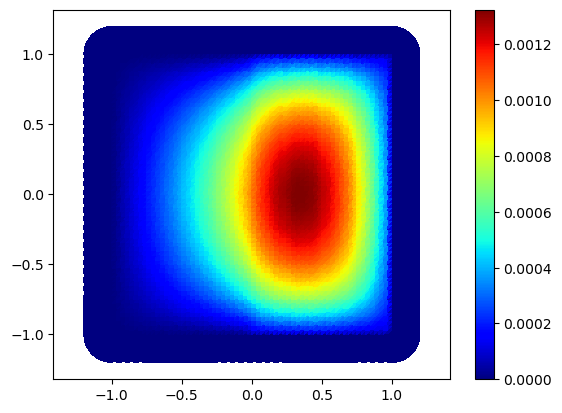}
        \caption{Left-Right}
    \end{subfigure}
    \hfill
    \begin{subfigure}[b]{0.32\textwidth}
        \includegraphics[width=\textwidth]{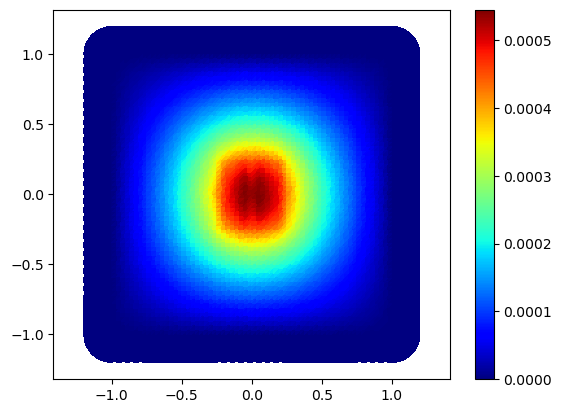}
        \caption{Inclusion}
    \end{subfigure}
      \begin{subfigure}[b]{0.32\textwidth}
        \includegraphics[width=\textwidth]{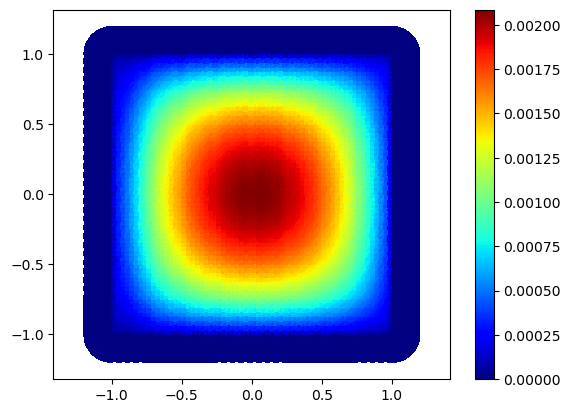}
        \caption{Fully nonlocal}
    \end{subfigure}
    \caption{Error for the 2D patch test for the two different coupling configurations and \(\mathbb{P}_{1}-\mathbb{P}_{1}\) discretization and the fully nonlocal problem with \(\mathbb{P}_{1}\) discretization.}
    \label{fig:patch-2d}
\end{figure}

Results of optimization based coupling using the same configurations are shown in \cref{tab:2d-opt-stats}.
The objective function effectively reaches zero as predicted by \cref{thm:main-theorem}.
In \cref{fig:2d-splice-opt-error}, we show the difference between the solutions of splice and optimization-based LtN coupling.
The different methods are numerically equivalent as the magnitude of the error can be attributed to either optimization error or quadrature error. 

\begin{table}[ht]
    \centering
    \begin{tabular}{|c|cc|}
    \hline
         &  Objective Function Value &BFGS Iterations \\
         \hline
        Left/Right &  5.871e-11 & 30  \\ 
        Inclusion  &  2.192e-10 & 32 \\ 
        \hline
    \end{tabular}
    \caption{Statistics of the optimization procedure on the 2D quadratic patch test on two different domain configuration for \(\mathbb{P}_{1}-\mathbb{P}_{1}\) discretizations}
    \label{tab:2d-opt-stats}
\end{table}

\begin{figure}
    \centering
    \begin{subfigure}[b]{0.32\textwidth}
        \includegraphics[width=\textwidth]{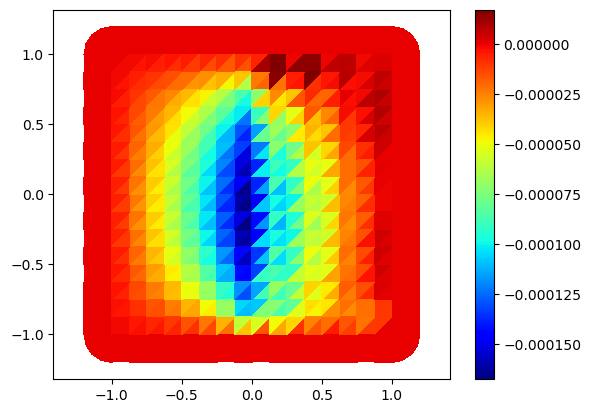}
          \label{Left-Right}
    \end{subfigure}
    \begin{subfigure}[b]{0.32\textwidth}
        \includegraphics[width=\textwidth]{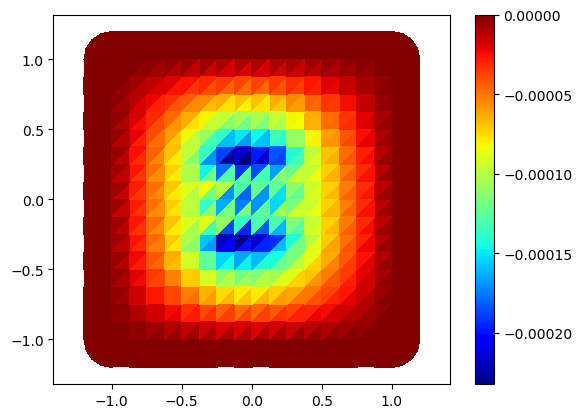}
        \label{Inclusion}
    \end{subfigure}
    \caption{Difference between solutions obtained using the splice and optimization methods for \(\mathbb{P}_{1}-\mathbb{P}_{1}\) discretization.}
    \label{fig:2d-splice-opt-error}
\end{figure}

Finally, we also show the $\mathbb{P}_0-\mathbb{P}_1$ coupling in 2D for the two different configurations. 
We use the same patch test problem, kernel, and domain configurations as before but with a slightly smaller mesh size $h \approx 0.02$.
The errors resulting from the coupling are shown in \cref{fig:patch-2d-p0-p1}.
Note that while the splicing errors are larger than in the matching mesh case, they are smaller than for the fully nonlocal problem.
This behavior is due to the lower rate of convergence of the \(\mathbb{P}_{0}\) discretization.

\begin{figure}
    \centering
    \begin{subfigure}[b]{0.32\textwidth}
        \includegraphics[width=\textwidth]{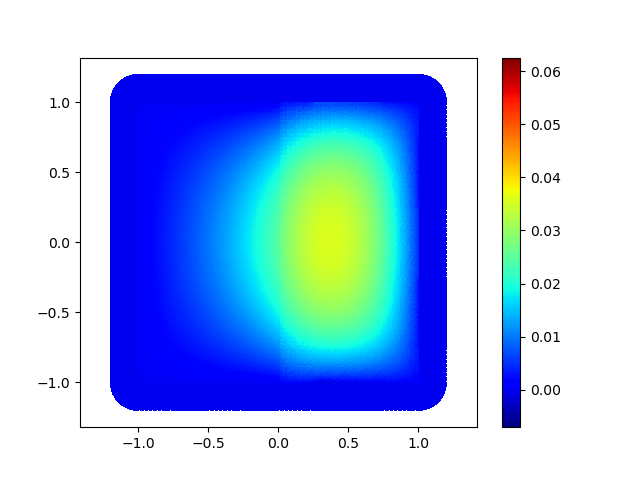}
        \caption{Left-Right}
    \end{subfigure}
    \begin{subfigure}[b]{0.32\textwidth}
        \includegraphics[width=\textwidth]{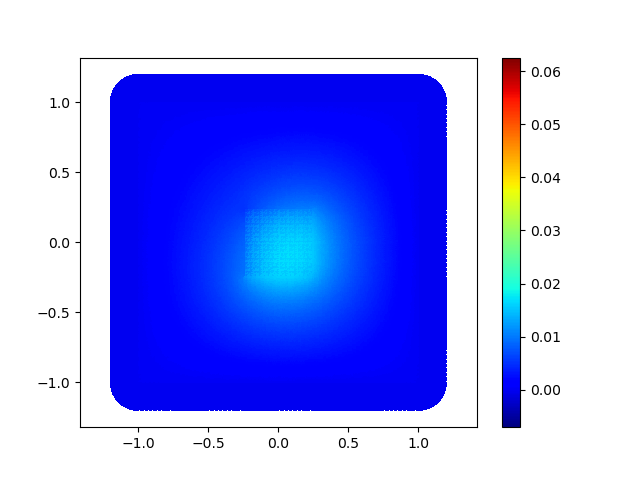}
        \caption{Inclusion}
    \end{subfigure}
      \begin{subfigure}[b]{0.32\textwidth}
        \includegraphics[width=\textwidth]{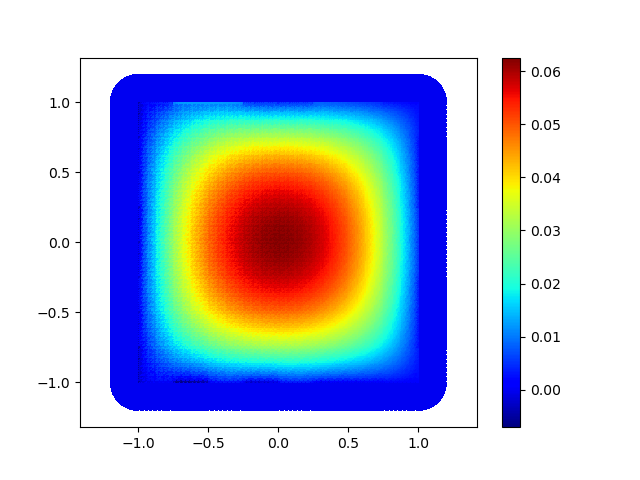}
        \caption{Full nonlocal}
    \end{subfigure}
    \caption{Error for the 2D patch test for the two different coupling configurations and \(\mathbb{P}_{1}-\mathbb{P}_{0}\) discretization and the fully nonlocal problem with \(\mathbb{P}_{0}\) discretization.}
    \label{fig:patch-2d-p0-p1}
\end{figure}

\subsection{Convergence to Local Models}
We turn to the question of whether our splicing method preserves the asymptotic compatibility property where the nonlocal operator $\mathcal{L}_\nonlocal \to \mathcal{L}_\local$ as $\delta \to 0$. 
We consider the \emph{local} problem $-\Delta u = f$ on $\Omega = (-1, 1)$ with forcing function
\begin{align}
  f(x) = \frac{1}{\abs{x}^{1/4}} + \sin(x)
\end{align}
and boundary conditions $u(-1) = -1, u(1) = 1$. 
The key property we wish to replicate is that \cref{eqn:nonlocal-model-full} with the same data will approach the same local solution as $\delta \to 0$. 

For the splicing configuration, we split the domain into left-right configuration where $\Omega_L = (-1, 0)$. 
On the nonlocal domain, we use the fractional kernel with $s = 0.25$ with varying horizon $\delta$. 
In \cref{fig:delta-conv-fig}, we plot the solutions of the splice equation with varying $\delta$ against the local solution.
We observe that the spliced solution converges to the local solution.
This can be seen in greater detail in \cref{fig:delta-conv-error}, where we plot the convergence rate of the $L^2$ error against $\delta$ for both the splicing solution and the fully nonlocal solutions.
Here, we see that the same $\mathcal O(\delta)$ convergence is seen for the spliced solution as if we had used a fully nonlocal method.

\begin{figure}
  \centering
  \begin{subfigure}[b]{0.49\textwidth}
    \includegraphics[width=\textwidth]{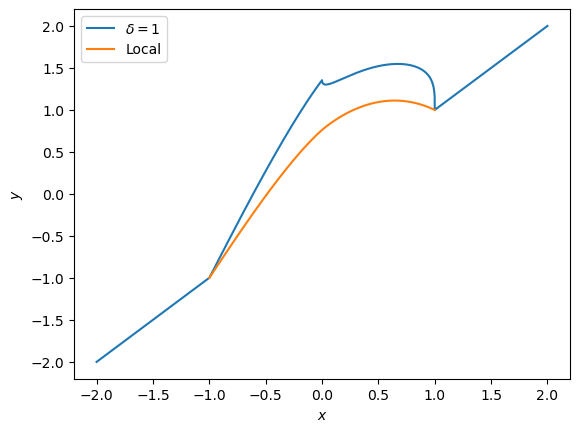} 
    \caption{$\delta = 1$}
  \end{subfigure}
    \begin{subfigure}[b]{0.49\textwidth}
    \includegraphics[width=\textwidth]{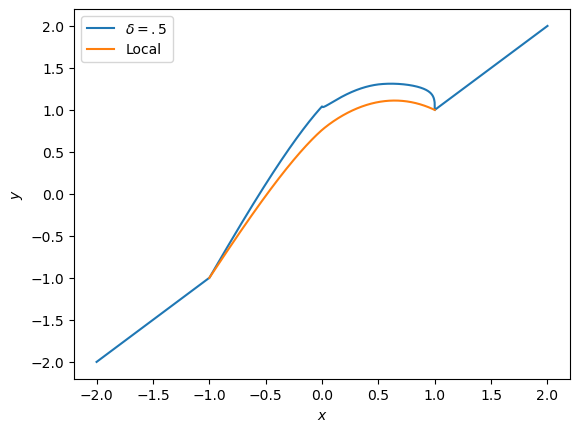} 
    \caption{$\delta = .5$}
  \end{subfigure} \\
  \begin{subfigure}[b]{0.49\textwidth}
    \includegraphics[width=\textwidth]{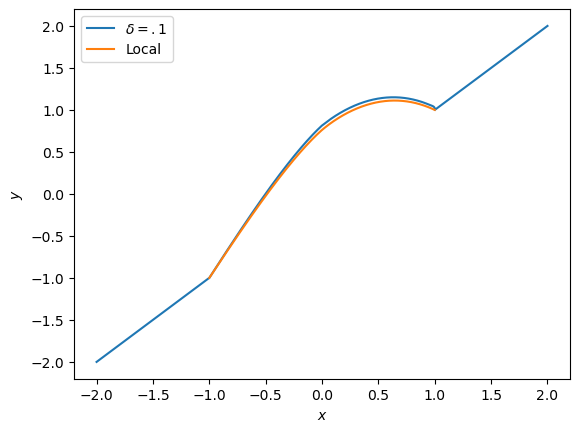} 
    \caption{$\delta = .1$}
  \end{subfigure}
  \begin{subfigure}[b]{0.49\textwidth}
    \includegraphics[width=\textwidth]{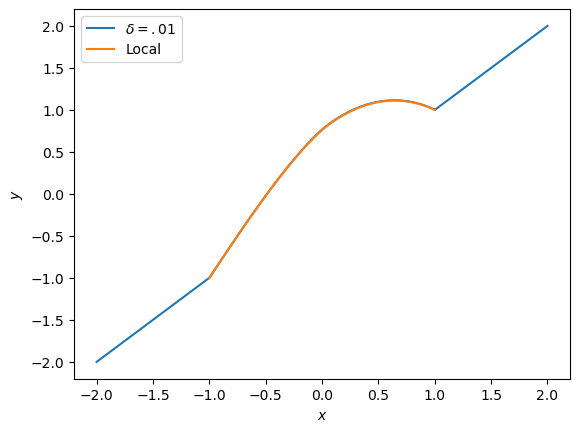} 
    \caption{$\delta = .01$}
  \end{subfigure}
    \caption{Plot of solutions obtained from applying to the splice method with a left-right splitting to increasingly smaller horizon.
    It is clear that the spliced solution approaches the local solution. 
    Note that for simplicity, we use the same mesh.}
    \label{fig:delta-conv-fig}
\end{figure}

\begin{figure}[tb]
  \centering
  \includegraphics[width=.5\textwidth]{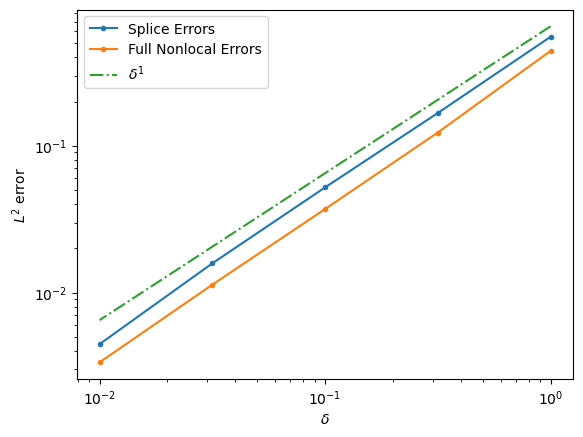}
  \caption{Plot of the $L^2$ errors between the local solution and the spliced solutions with varying $\delta$. 
  We observe a convergence rate of roughly $\delta$.}
  \label{fig:delta-conv-error}
\end{figure}

\subsection{Solutions with Discontinuities}

\paragraph{1D Examples}
We now examine solutions to \cref{eqn:nonlocal-model-full} where discontinuities may occur in a small region with smooth solutions outside that area.
This is a typical use case for LtN coupling.
Consider \cref{eqn:nonlocal-model-full} on $\Omega = (-1, 1)$ with forcing function
\begin{align}    
f(x) = \begin{cases}
        0 & x \in (-1,-\delta)\cup(\delta,1), \\
        -\frac{2}{\delta^2} (-\frac{1}{4}(\log(\delta) - \log(-x)) & x\in [-\delta, 0), \\
        \frac{2}{\delta^2} (\frac{1}{4} (\log (x) - \log(\delta)) & x\in [0,\delta]
    \end{cases}
\end{align}
and the inverse distance kernel of horizon $\delta$.
It can be shown that the analytic solution is
\begin{align}
    u(x) = \begin{cases}
        x + \frac{1}{4} & x < 0, \\
        x & x \ge 0
    \end{cases}
\end{align}
meaning that the solution has a jump at $x = 0$ but is linear otherwise.

We exhibit the effectiveness of the splicing method by letting $\delta = 0.1$ and $\Omega_L = (-1, 1) \setminus [-.25, .25]$ in both the $\mathbb{P}_1-\mathbb{P}_1$ and $\mathbb{P}_0-\mathbb{P}_1$ coupling.
In \cref{fig:sing-1d}, we plot the solution arising from the splice LtN alongside the analytic solution for both the $\mathbb{P}_1-\mathbb{P}_1$ and $\mathbb{P}_0-\mathbb{P}_1$ coupling.
With the exception of some numerical artifacts typical of discontinuous solutions in the $\mathbb{P}_1$ case, we see that the splice method works well. 

\begin{figure}
    \centering
    \begin{subfigure}[t]{0.48\textwidth}
        \includegraphics[width=\textwidth]{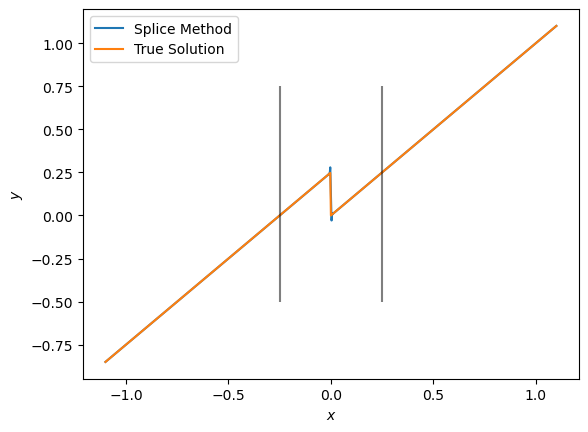}
        \caption{Solution using $\mathbb{P}_1-\mathbb{P}_1$ coupling; bars indicate the nonlocal region. }
    \end{subfigure}
    \hfill
    \begin{subfigure}[t]{0.48\textwidth}
        \includegraphics[width=\textwidth]{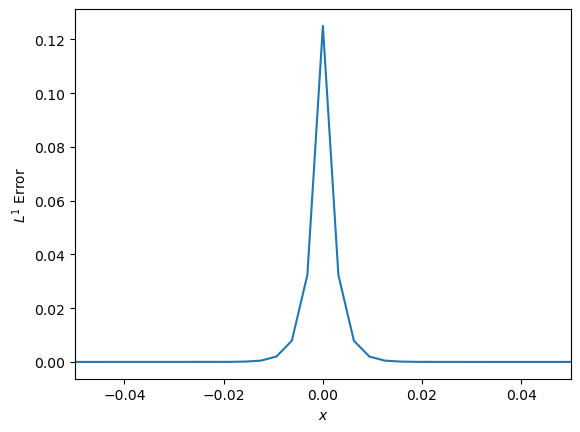} 
        \caption{$\mathbb{P}_1-\mathbb{P}_1$ Error}
    \end{subfigure}\\
        \begin{subfigure}[t]{0.48\textwidth}
        \includegraphics[width=\textwidth]{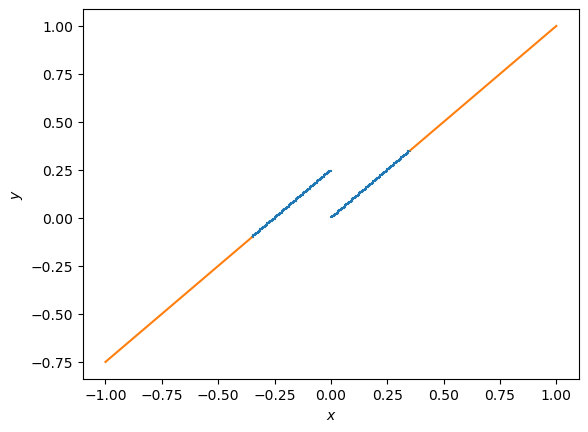}
        \caption{Solution using $\mathbb{P}_0-\mathbb{P}_1$ coupling; blue/orange indicates $\mathbb{P}_0, \mathbb{P}_1$ respectively.}
    \end{subfigure}
    \hfill
    \begin{subfigure}[t]{0.48\textwidth}
        \includegraphics[width=\textwidth]{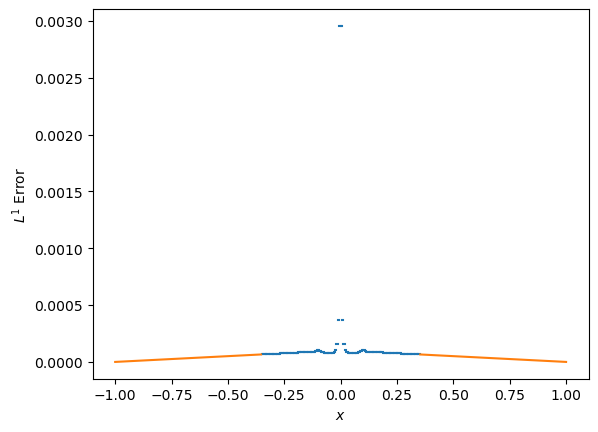}
        \caption{$\mathbb{P}_0-\mathbb{P}_1$ Error}
    \end{subfigure}
    \caption{Figure of 1D singular solutions and errors.}
    \label{fig:sing-1d}
\end{figure}

\paragraph{2D Examples}
Let $\vec{x}^\star \in \mathbb{R}^2$ be some arbitrary point, and $r > 0$ a chosen parameter.
Consider the following forcing function 
\begin{align}    
f(\vec{x}) = c_\delta\begin{cases}
        0 & \norm{\vec{x} - \vec{x}^\star} > r + \delta \\
        A(\delta, r, \norm{\vec{x} - \vec{x}^\star}) - \pi \delta^2 & \norm{\vec{x} - \vec{x}^\star} < r \\
        A(\delta, r, \norm{\vec{x} - \vec{x}^\star}) & \text{otherwise}
    \end{cases}
\end{align}
where $c_\delta$ a scaling constant, $A(R,  r, \norm{\vec{x} - \vec{x}^\star})$ is the area of the intersection between two circles of radii $R, r$ and centers $\norm{\vec{x} - \vec{x}^\star}$ away
\begin{align*}
    A(R, r, d) := \begin{cases}
        \pi \min(R, r)^2 & \abs{R - r} > d\\
        k(R, r, d) & R + r > d \\
        0 & \text{otherwise}
    \end{cases}
\end{align*}
with 
\begin{align*}
  k(R, r, d) &= r^2 \arccos\left( \frac{d^2 + r^2 - R^2}{2dr} \right) + R^2 \arccos\left( \frac{d^2 + R^2 - r^2}{2dR} \right) \\
  &\qquad- \frac{1}{2} \sqrt{(-d + r +R)(d + r - R)(d - r + R)(d + r + R)}.
\end{align*}
It is easy to show that the true solution is simply a cylinder indicator
\begin{align}
    u(\vec{x}) = \chi_{B_{r}(\vec{x}^\star)}(\vec{x}).
\end{align}

We let $\vec{x}^\star = (0, 0)$, $r = 0.2$ and $\delta = 0.25$.
For the splicing domains, we first show the effects of varying the size of $\Omega_N$ on the resulting solution. 
In \cref{fig:2dsing-window}, we show three different sizes of $\Omega_N = \{(-.2, .2)^2,  (-.3, .3)^2, (-.4, .4)^2\}$.
The first two windows are clearly too small, failing to capture the correct magnitude of the jump and introducing additional artifacts.
However, the last case, even though the domain is smaller than the sum of the radius of discontinuity radius of the kernel, effectively captures the solution.
This is reflected in \cref{fig:2dsing-error} where we show the effect of changing the window size on the error.
Intuitively, this makes sense: we expect that we need to capture a region of width \(\delta\) around the discontinuity surface at \(\norm{\vec{x}}=r\) using the nonlocal domain, meaning that all nonlocal effects should be captured by nonlocal domains \((-a,a)^{2}\) for \(a\geq0.45\).

\begin{figure}
  \centering
  \begin{subfigure}[b]{0.4\textwidth}
    \includegraphics[width=\textwidth]{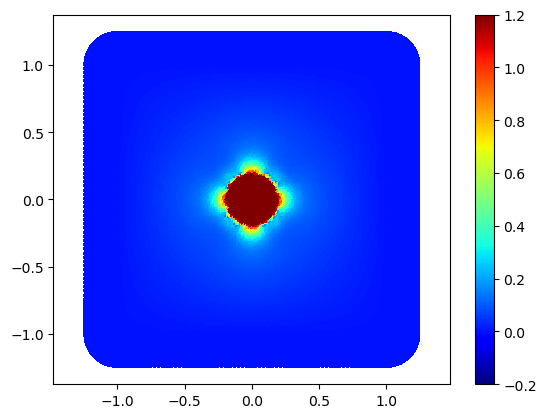}
    \caption{$\Omega_N = (-.2, .2)^2$}
  \end{subfigure}%
  \begin{subfigure}[b]{0.4\textwidth}
    \includegraphics[width=\textwidth]{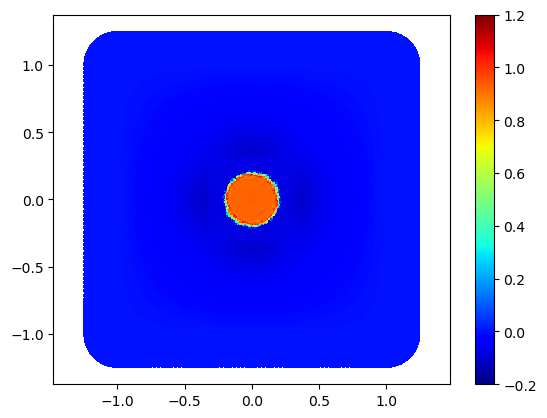}
    \caption{$\Omega_N = (-.3, .3)^2$}
  \end{subfigure}
  \begin{subfigure}[b]{0.4\textwidth}
    \includegraphics[width=\textwidth]{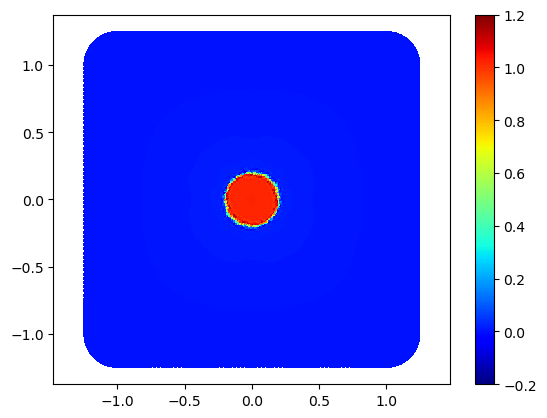}
    \caption{$\Omega_N = (-.4, .4)^2$}
  \end{subfigure}%
  \begin{subfigure}[b]{0.4\textwidth}
    \includegraphics[width=\textwidth]{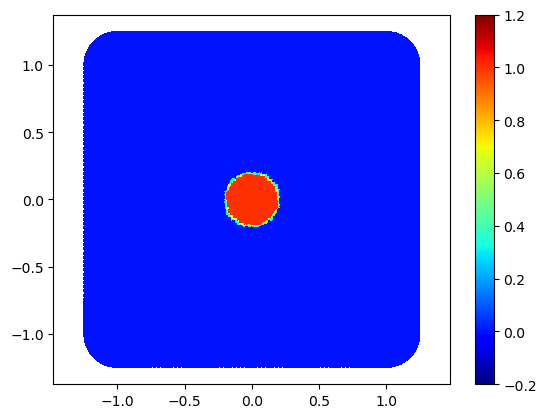}
    \caption{Analytic solution}
  \end{subfigure}
  \caption{Plots of the solutions computed using varying $\Omega_N$ sizes.
  Note that if the window is too small, the discontinuities are not captured correctly and additional artifacts are introduced.}
  \label{fig:2dsing-window}
\end{figure}

\begin{figure}
  \centering
  \includegraphics[width=.5\textwidth]{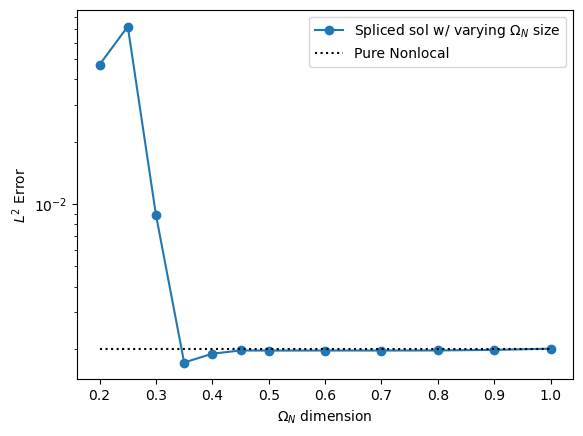}
  \caption{
    Plot showing effect of window size on $L^2$ error.
    The $x$-axis is half of the dimension of $\Omega_N$ (e.g. a value of 0.2 corresponds to $\Omega_N = (-0.2, 0.2)^2$).
    Note that at $\Omega_N = (-0.35, 0.35)^2$, the spliced solution is smaller than the one with pure nonlocal due to numerical issues associated with sharp jumps. 
  }
  \label{fig:2dsing-error}
\end{figure}

\subsection{A Time Dependent Problem}
One considerable benefit hinted above for the splice LtN coupling is its computational cost.
Since no optimization needs to be done or intrusive alteration of the code, one can easily incorporate the splicing method into time-dependent problems. 

We consider the nonlocal heat equation with a forcing term
\begin{align}\label{eqn:time-dynamics}
  u_t - 0.1 \mathcal{L}_N u &= f
\end{align}
on $\Omega = (-1, 1)^2$ with homogeneous Dirichlet volume conditions.
We use a constant kernel with horizon $\delta = 0.2$.
The forcing function is 
\begin{align}\label{eqn:time-forcing}
  f(\vec{x}, t) = \chi_{B_{0.1}(\overline{\vec{x}}(t))}(\vec{x})
\end{align}
where $\overline{\vec{x}}(t) = \left(\frac{1}{2}\cos(t), \frac{1}{2} \sin(t)\right)$, meaning the forcing term is simply a ball of radius $0.1$ moving in a circular fashion around the domain.
Let the initial condition be $u(\vec{x}, 0) = \chi_{B_{0.1}(\overline{\vec{x}}(0))}(\vec{x})$.

We use a simple backward Euler discretization in time, resulting in the linear system
\begin{align}\label{eqn:time-stepping-full}
  \left(\vec M + 0.1\Delta t \lapl\right) \vec u^{n+1} = \vec M \vec u^n + \Delta t \vec f,
\end{align}
where $\vec M$ is the classical mass matrix.
For this particular example, we let $\Delta t = 0.1$ and perform the calculations on a mesh with $h \approx 0.04$ resulting in a mesh with 16,129 dofs in the interior.
In order to obtain a reference solution, we uniformly refine the mesh twice and solve the fully nonlocal problem with 261,121 dofs.
The solution of \cref{eqn:time-dynamics} results in a moving circular region with mild singularities; see \cref{fig:time-full} for plots of the solution computed using \cref{eqn:time-stepping-full} the full nonlocal matrix on the fine mesh. 

\begin{figure}
  \centering
  \begin{subfigure}[b]{0.4\textwidth}
    \includegraphics[width=\textwidth]{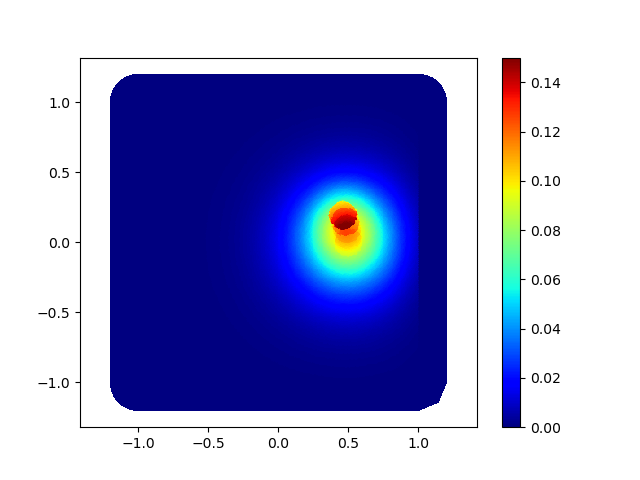}
    \caption{$t = .5$}
  \end{subfigure}%
  \begin{subfigure}[b]{0.4\textwidth}
    \includegraphics[width=\textwidth]{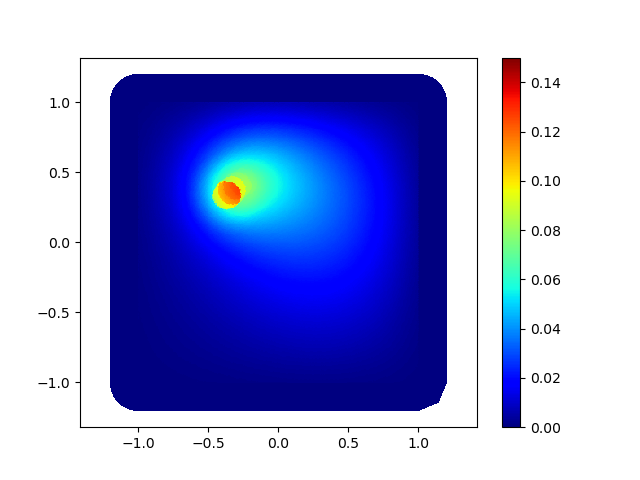}
    \caption{$t = 2.5$}
  \end{subfigure}
  \begin{subfigure}[b]{0.4\textwidth}
    \includegraphics[width=\textwidth]{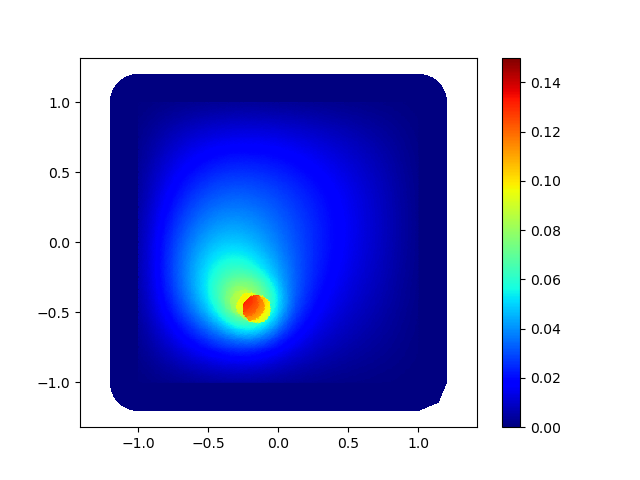}
    \caption{$t = 4.5$}
  \end{subfigure}%
  \begin{subfigure}[b]{0.4\textwidth}
    \includegraphics[width=\textwidth]{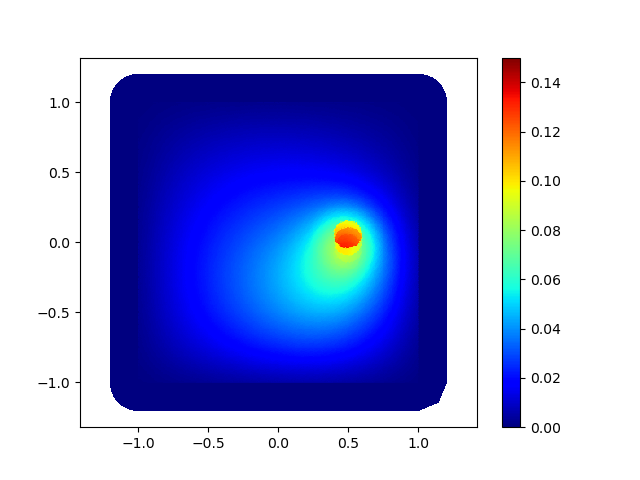}
    \caption{$t = 6.5$}
  \end{subfigure}
  \caption{Plots of the solutions computed using the full nonlocal Laplacian.}
  \label{fig:time-full}
\end{figure}

We consider three different methods of using the splice LtN coupling for such problems.
The first is to track the region where the forcing is nonzero, and utilize a local Laplacian for the remaining domain.
Given \cref{eqn:time-forcing} and the value of \(\delta\), we set the nonlocal domain as
\begin{align*}
  \Omega_\nonlocal := \overline{\vec{x}}(t) + (-0.3,0.3)^{2}
\end{align*}
consisting of a square with side length \(0.6\) centered at $\overline{\vec{x}}(t)$.
In the second approach, we use a moving domain as in the first approach, but augment it with an \emph{internal} boundary layer
\begin{align*}
  \Omega_\nonlocal := (-1,1)\setminus(-1+\delta,1-\delta)^{2} \cup \left(\overline{\vec{x}}(t) + (-0.3,0.3)^{2}\right).
\end{align*}
This is motivated by noticing that the errors arising from this particular problem lie near the boundary.
The last approach consists of fixing the nonlocal domain.
This results in a larger nonlocal region, but with the simplicity of having to solve the same linear system in each time step.
In this case, we choose
\begin{align*}
  \Omega_\nonlocal = \left(-.8 , .8 \right)^2 \setminus  \left[-.2 , .2 \right]^2
\end{align*}
resulting in a square, annular region.
Regardless of the approach, the resulting LtN time-stepping scheme would be
\begin{align}\label{eqn:time-stepping-splice}
  \left(\vec M + \frac{\Delta t }{10}\matr A^S \right) \vec u^{n+1} = \vec M \vec u^n + \Delta t \vec f
\end{align}
where $\matr A^S$ is the splicing matrix.
In the case of time-dependent coupling regions \(\matr A^S\) changes from time step to time step.

In \cref{fig:time-err}, we plot the $L^2$-error between the reference solution and the three splice approaches and the fully nonlocal equation.
The performance of the constant window and the simple moving window are quite similar. 
Surprisingly, the introduction of a internal boundary layer drastically reduces the error.
This suggests that the choice of nonlocal domains in the splitting is a nontrivial endeavor as the forcing function is in fact zero near the boundary.
The appropriate choice of nonlocal/local domains is generally not trivial; this is the subject of future work.

We plot the actual $L^1$-error for a specific time in \cref{fig:time-error-comparison}; note that for the simple window and constant window examples, a significant error lies around the boundary.
Nevertheless, examining the plot of the solution at the final time $t = 10$ in \cref{fig:time-last-comparison}, we see that the jump is clearly defined for all the splicing methods.

\begin{figure}[tb]
  \centering
  \includegraphics[width=.5\textwidth]{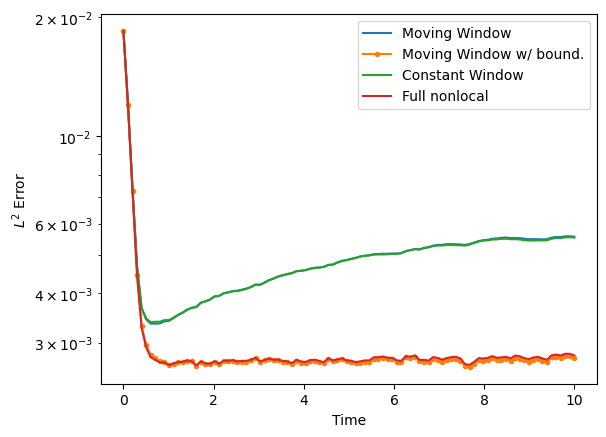}
  \caption{Plot of the $L^2$ error with respect to time for the three different splicing configurations and the full nonlocal method.
  The moving and constant window splicing method have roughly the same errors, though adding a internal boundary layer greatly reduces the error. 
  }
  \label{fig:time-err}
\end{figure}

\begin{figure}
  \centering
  \begin{subfigure}[t]{0.4\textwidth}
    \includegraphics[width=\textwidth]{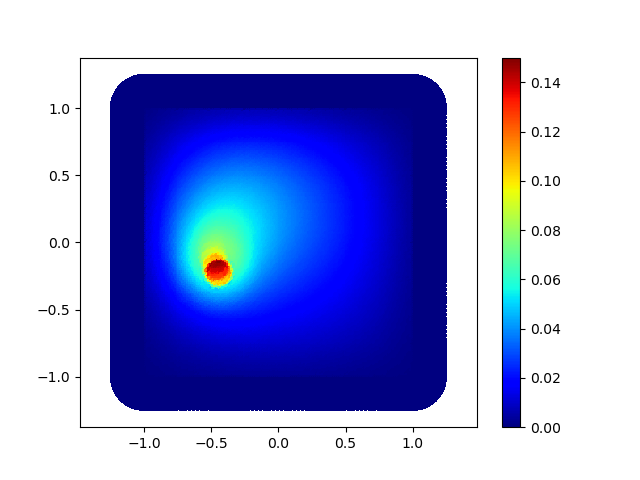}
    \caption{Fully nonlocal}
  \end{subfigure}%
  \begin{subfigure}[t]{0.4\textwidth}
    \includegraphics[width=\textwidth]{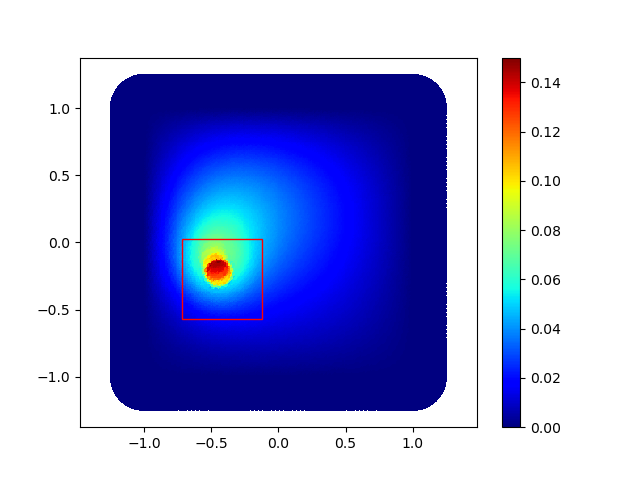}
    \caption{splice, moving domains}
  \end{subfigure} \\%
  \begin{subfigure}[t]{0.4\textwidth}
    \includegraphics[width=\textwidth]{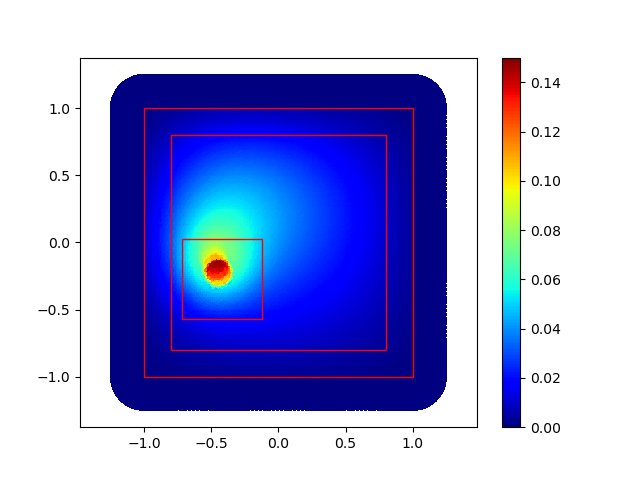}
    \caption{splice, moving domain with boundary layer}
  \end{subfigure}%
    \begin{subfigure}[t]{0.4\textwidth}
    \includegraphics[width=\textwidth]{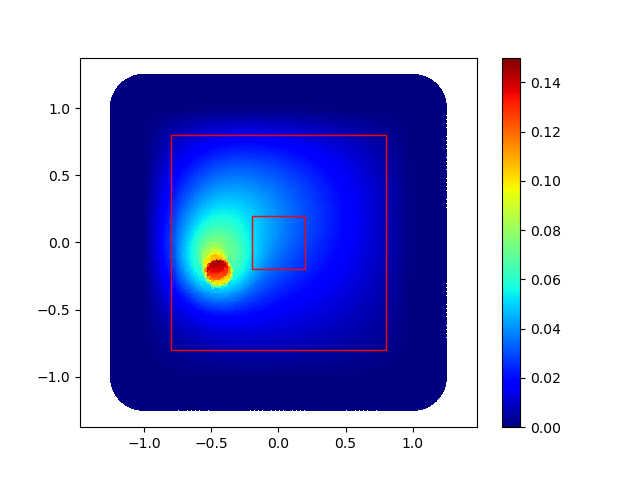}
    \caption{splice, constant domain}
  \end{subfigure}%
  \caption{Plots of the solutions at \(t=10\).
  The red outlines indicate the nonlocal domain $\Omega_\nonlocal$.}
  \label{fig:time-last-comparison}
\end{figure}

\begin{figure}
  \centering
  \begin{subfigure}[t]{0.4\textwidth}
    \includegraphics[width=\textwidth]{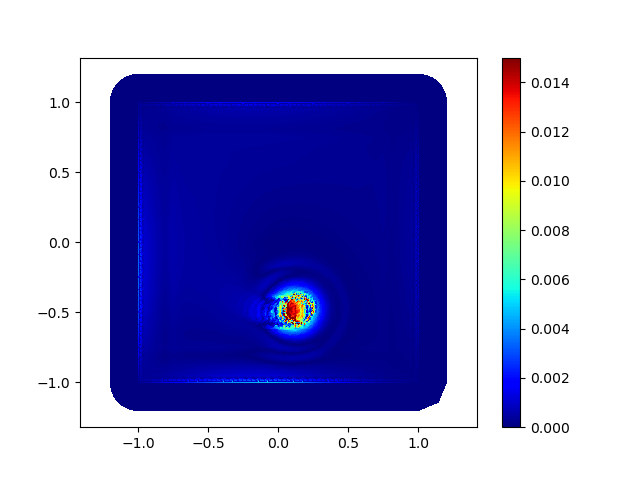}
    \caption{Fully nonlocal}
  \end{subfigure}%
  \begin{subfigure}[t]{0.4\textwidth}
    \includegraphics[width=\textwidth]{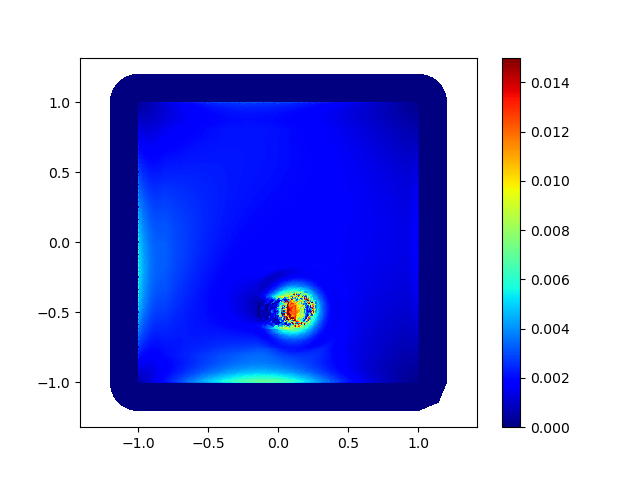}
    \caption{splice, moving domains}
  \end{subfigure} \\%
  \begin{subfigure}[t]{0.4\textwidth}
    \includegraphics[width=\textwidth]{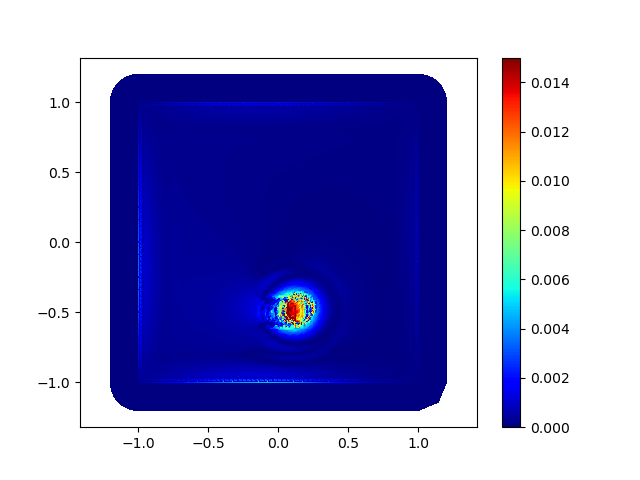}
    \caption{splice, moving domain with boundary layer}
  \end{subfigure}%
    \begin{subfigure}[t]{0.4\textwidth}
    \includegraphics[width=\textwidth]{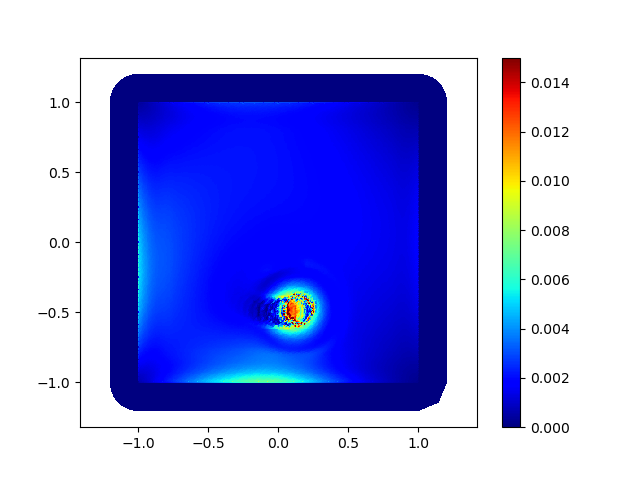}
    \caption{splice, constant domain}
  \end{subfigure}%
  \caption{Plots of the $L^1$ error at \(t=5\).
  The red outlines indicate the nonlocal domain $\Omega_\nonlocal$. 
  The error congregates around where the forcing term is nonsmooth, and the boundary.}
  \label{fig:time-error-comparison}
\end{figure}

\section{Conclusion}

We have presented a splice method for coupling nonlocal and local diffusion models in weak form.
The analysis shows that the coupling is well posed at the discrete level and inherits properties from the optimization-based approach.
The method is straightforward to implement and the arising system of equations is easier to solve than the related more general optimization-based coupling.
A sequence of numerical examples have illustrated that the method passes patch tests and is quite adaptable to different use cases.
The final time dependent test case has led to an important question.
While the method we presented allows the coupling of arbitrary splittings into local and nonlocal domains, it is not obvious a priori how to choose a specific splitting that will minimize the error while being computationally efficient.
We will address this question in future work.

\section{Funding}

Sandia National Laboratories is a multimission laboratory managed and operated by National Technology \& Engineering Solutions of Sandia, LLC, a wholly owned subsidiary of Honeywell International Inc., for the U.S. Department of Energy’s National Nuclear Security Administration under contract DE-NA0003525.

This paper describes objective technical results and analysis. Any subjective views or opinions that might be expressed in the paper do not necessarily represent the views of the U.S. Department of Energy or the United States Government.

SAND No: SAND2024-04697O

\appendix

\section{Technical Proofs}
\subsection{Proof of \texorpdfstring{\cref{thm:main-theorem}}{Existence Theorem} }
We proceed much as \cite{DEliaBochev2021_FormulationAnalysisComputationOptimization,d2016coupling} and decompose the nonlocal and local equations into discrete harmonic and homogeneous components:
\begin{align}
  \vec u_\nonlocal(\vec \theta_\nonlocal) := \vec v_\nonlocal(\vec \theta_\nonlocal) + \vec w_\nonlocal,\qquad 
  \vec u_\local(\vec \theta_\local) := \vec v_\local(\vec \theta_\local) + \vec w_\local
\end{align}
where $\vec v_\nonlocal, \vec v_\local$ are the discrete harmonics which solves
\begin{align*}
  \lapl[\nonlocal, \nonlocal] \vec v_N(\vec \theta_N) = - \lapl[\nonlocal, (\interactionSolve)] \vec \theta_N, \qquad \lapl[\local, \local] \vec v_\local(\vec \theta_\local) = -  \lapl[\local, \localBoundSolve] \vec \theta_\local
\end{align*}
while $\vec w_{\cdot}$ is the forcing terms, which are independent of the control functions
\begin{align*}
  \lapl[\nonlocal, \nonlocal] \vec w_N(\vec \theta_N) = \vec f_\nonlocal -  \lapl[\nonlocal, \interactionGiven] \vec g_{\interactionGiven}, \qquad \lapl[\local, \local] \vec w_\local(\vec \theta_\local) =  \vec f_\local - \lapl[\local, \localBoundGiven] \vec g_{\localBoundGiven}.
\end{align*}
For notational purposes, we will use the shorthand $\vec v_{\cdot} := \vec v_{\cdot}(\vec \theta_{\cdot})$.

We also define the continuous decomposition: for $\theta_{\local} \in H^{1/2}_{0}(\localBoundSolve)$, and 
$\theta_{\nonlocal} \in \{ u \in \operatorname{trace}_{\Omega_\interactionSolve}V(\Omega_{\nonlocal}\cup\Omega_{\interactionSolve}\cup\Omega_{\interactionGiven} ) \mid u|_{\Omega_\interactionGiven} = 0\}$,
\begin{align*}
  u_\nonlocal = v_\nonlocal(\theta_\nonlocal) + w_\nonlocal, \qquad u_\local = v_\local(\theta_\local) + w_\local
\end{align*}
where
\begin{align*}
    \left\{
    \begin{aligned}
      -\mathcal{L}_\nonlocal v_{\nonlocal}(\vec x) &= 0, && \vec x \in \Omega_{\nonlocal}, \\
      v_\nonlocal(\vec x) &= \theta_\nonlocal(\vec x), && \vec x \in \Omega_{\interactionSolve}, \\
      v_\nonlocal(\vec x) &= 0,  && \vec x \in\Omega_{\interactionGiven},
    \end{aligned}
    \right.
    \qquad
      \left\{
\begin{aligned}
    -\mathcal{L}_\local  v_\local(\vec x) &= 0, && \vec x \in \Omega_{\local}, \\
     v_\local(\vec x) &=  \theta_\local(\vec x), && \vec x \in \localBoundSolve, \\
     v_\local(\vec x) &= 0, && \vec x \in \localBoundGiven.
\end{aligned}
\right.
\end{align*}
and
\begin{align*}
    \left\{
    \begin{aligned}
      -\mathcal{L}_\nonlocal w_{\nonlocal}(\vec x) &= f(\vec x), && \vec x \in \Omega_{\nonlocal}, \\
      w_\nonlocal(\vec x) &= 0, && \vec x \in \Omega_{\interactionSolve}, \\
      w_\nonlocal(\vec x) &= 0,  && \vec x \in\Omega_{\interactionGiven},
    \end{aligned}
    \right.
    \qquad
      \left\{
\begin{aligned}
    -\mathcal{L}_\local  w_\local(\vec x) &= f(\vec x), && \vec x \in \Omega_{\local}, \\
     w_\local(\vec x) &= 0, && \vec x \in \localBoundSolve, \\
     w_\local(\vec x) &= 0, && \vec x \in \localBoundGiven.
\end{aligned}
\right.
\end{align*}

Before proceeding, we state the necessary preliminaries that we need, starting with maximum principle results regarding the local/nonlocal Laplacians. 
\begin{lemma}[Strong Local Maximum Principle]\label{thm:strong-max-local-lap} 
  Let \(\theta_{\local}\in C^{0}(\localBoundSolve)\) and assume that \(\Omega_{\local}\) satisfies an exterior cone condition.
  Then \(v_{\local}\in C^{\infty}(\Omega_{\local})\cap C(\overline{\Omega_{\local}})\).
  If \(v_{\local}\) attains an extremum at a point in \(\Omega_{\local}\) then \(v_{\local}\) is constant.
\end{lemma}
\begin{proof}
  \cite[Theorem 6.13]{GilbargTrudinger2001} states that \(v_{\local}\in C^{2}(\Omega_{\local})\cap C(\overline{\Omega_{\local}})\).
  \cite[\S 6.3.1, Theorem 3]{evans2022partial} gives \(v_{\local}\in C^{\infty}(\Omega_{\local})\).
  The statement of the strong maximum principle can be found in \cite[\S 6.4.2, Theorem 3]{evans2022partial}.
\end{proof}

\begin{lemma}[Weak maximum principle for the nonlocal problem]\label{thm:max-nonlocal-lap}
  Suppose the symmetric kernel $\gamma(\vec x, \vec y)$ satisfies
  \begin{enumerate}
    \item Two integrability conditions
    \begin{align*}
      \sup_{\vec x \in \mathbb{R}^n} \int_{\mathbb R^n} \min(1, \abs{\vec x - \vec y}^2) \gamma(\vec x, \vec y) \, d\vec y < \infty, 
    \end{align*}
    and 
    \begin{align*}
      \sup_{\vec x \in \mathbb{R}^n} \int_{\mathbb R^n} \min(1, \abs{\vec z}^2) \abs{\gamma(\vec x, \vec x + \vec z) - \gamma(\vec x, \vec x - \vec z)} \, d\vec z < \infty, 
    \end{align*}
    \item Nontriviality condition: for all \(r\), the function \(j(\vec z) = \operatorname{essinf} \{\gamma(\vec x, \vec x \pm \vec y)  \mid \vec x \in \mathbb{R}^n \} \) does not vanish identically on $B_r(\vec 0)$. 
  \end{enumerate}
  Then, $v_\nonlocal$ satisfies the weak maximum principle
  \begin{align*}
    \sup_{\vec{x}\in\Omega_{\nonlocal}}v_{\nonlocal}(\vec{x}) \leq \max\{0, \sup_{\vec{x}\in\Omega_{\interactionSolve}}\theta_{\nonlocal}(\vec{x})\}.
  \end{align*}
  where the suprema are taken in the almost everywhere sense.
\end{lemma}
\begin{remark}
  By definition $v_\nonlocal$ has homogeneous zero boundary condition in $\Omega_{\interactionGiven}$, hence the 0 in the above inequality. 
\end{remark}
\begin{remark}
  The above assumptions are fairly weak, and are trivially satisfied by the kernels \eqref{eq:fracKernel} and \eqref{eq:integrableKernel} that we used for the numerical examples.
\end{remark}
\begin{proof}
  Let $\bar v_{\nonlocal} = v_{\nonlocal} - \Lambda$ where $\Lambda:= \sup_{\vec{x}\in\Omega_{\interactionSolve} \cup \Omega_{\interactionGiven}} \theta_{\nonlocal}(\vec{x})_{+}\geq 0$ which also implies that $\theta_N(\vec x) - \Lambda \le 0$.
  In particular, if $\theta_N(\vec x) \le 0$ for all $\vec x \in \Omega_{\interactionSolve}$, then $\Lambda = 0$ and $\bar v_{\nonlocal} = v_{\nonlocal}$.
  The function $\bar v_{\nonlocal}$ is a weak solution to
  \begin{equation*}
  \left\{
    \begin{aligned}
      -\mathcal{L}_\nonlocal \bar v_{\nonlocal}(\vec x) &= 0, && \vec x \in \Omega_{\nonlocal}, \\
      \bar v_{\nonlocal}(\vec x) &= \theta_\nonlocal(\vec x) - \Lambda, && \vec x \in \Omega_{\interactionSolve}, \\
      \bar v_{\nonlocal}(\vec x) &= -\Lambda,  && \vec x \in\Omega_{\interactionGiven},
    \end{aligned}
    \right.
\end{equation*}
  Note that on $\Omega_{\interactionSolve}\cup\Omega_{\interactionGiven}$, $\bar v_{\nonlocal} \le 0$.
  Thus, using the assumptions on the kernel and applying \cite[Thm. 2.4]{JarohsWeth2019_StrongMaximumPrincipleNonlocalOperators}, we have that $\bar v_{\nonlocal} \le 0$ everywhere meaning that
  \begin{align*}
    \sup_{\Omega_\nonlocal} v_{\nonlocal} - \Lambda = \sup_{\Omega_\nonlocal} \bar v_{\nonlocal} \le 0 = \max\{0, \sup_{\Omega_{\interactionSolve} \cup \Omega_{\interactionGiven}} v_{\nonlocal} - \Lambda\}
  \end{align*}
  Finally, adding $\Lambda$ to both sides and using that \(v_{\nonlocal}=0\) on \(\Omega_{\interactionGiven}\) gives the result.
\end{proof}

We begin with a strengthened Cauchy-Schwarz inequality, which is slightly different from that of Lemma 4.3 of \cite{DEliaBochev2021_FormulationAnalysisComputationOptimization}:
\begin{lemma}[Continuous Strengthened Cauchy-Schwarz inequality]
    Let \(\theta_{\local} \in \Theta_L\) and 
    \(\theta_{\nonlocal} \in \operatorname{trace}_{\Omega_{\interactionSolve}} V(\Omega_{\nonlocal}\cup\Omega_{\interactionSolve}\cup\Omega_{\interactionGiven} )\),
  then there exists a constant \(0<c<1\) such that
  \begin{align*}
    \abs{\left(v_{\local}, v_{\nonlocal} \right)_{L^{2}(\Omega_{b})}} &\leq c \norm{v_{\local}}_{L^{2}(\Omega_{b})} \norm{v_{\nonlocal}}_{L^{2}(\Omega_{b})}.
  \end{align*}
\end{lemma}
\begin{remark}
  Note that the local control is from the space of finite element functions on the trace, meaning we are still allowed to apply \cref{thm:strong-max-local-lap}.
  The harmonic extension is taken in the continuous sense and not in discrete manner.
\end{remark}
\begin{proof}
  The proof proceeds exactly as \cite[Lemma 4.3]{DEliaBochev2021_FormulationAnalysisComputationOptimization} by using contradiction.
  In particular, if the strengthened result does not hold, then there exists $v_\local^*$, $v_\nonlocal^*$ such that
  \begin{align*}
    \abs{\left(v_{\local}^*, v_{\nonlocal}^* \right)_{L^{2}(\Omega_{b})}} = \norm{v_{\local}^*}_{L^{2}(\Omega_{b})} \norm{v_{\nonlocal}^*}_{L^{2}(\Omega_{b})}
  \end{align*}
  meaning that, by equality condition of Cauchy-Schwarz, $v^*_\local = k v_\nonlocal^*$ for some $k \in \mathbb{R}$. 
  We refer the reader to \cite{DEliaBochev2021_FormulationAnalysisComputationOptimization} for details regarding the rigorous construction of $v_\local^*$, $v_\nonlocal^*$.

  Thus, one simply has to show that if $v_\local^* = k v_\nonlocal^*$ for some $k \in \mathbb{R}$, then $v_\local^* = v_\nonlocal^* = 0$.
  Without loss of generality, assume that $\sup_{\Omega_b} v_\local^* \ge 0$.
  By the maximum principle, $v_\local^*$ cannot be constant on $\Omega_b$ unless $v_\local^* = 0$ everywhere.

  Applying the maximum theorems, we see that
  \begin{align*}
    \sup_{\localBoundSolve} v_\local^* &= \sup_{\localBoundSolve} k v_\nonlocal^* \stackrel{\localBoundSolve\subset \Omega_{\nonlocal}}{\le} \sup_{\Omega_\nonlocal} k v_\nonlocal^* \stackrel{\text{Lem.~\ref{thm:max-nonlocal-lap}}}{\le} \sup_{\Omega_\interactionSolve} kv_\nonlocal^* = \sup_{{\Omega}_\interactionSolve} v_\local^* \stackrel{\text{Lem.~\ref{thm:strong-max-local-lap}}}{<} \sup_{\localBoundSolve} v_\local^*
  \end{align*}
  where we note that $\Omega_{N, \interactionSolve}$ and $\Gamma$ are separated by a $\mathcal O(h)$ distance, hence, contradiction; for clarity, we refer to the reader to \cref{fig:max} for a schematic in 1D.
  Hence, $v_\local^*$ must be constant, and hence zero by the prescribed boundary conditions. 
\end{proof}
\begin{figure}[tb]
        \centering
      \begin{tikzpicture}[scale=.8]
          \draw[very thick] (0,0) -- (6.5,0);
    \draw[draw, blue, double=blue, double distance=2\pgflinewidth](0,3pt) -- (0,-3pt); 
    \draw[draw, red, double=red, double distance=2\pgflinewidth](3.2,3pt) -- (3.2,-3pt); 
    \draw[draw, red, double=red, double distance=2\pgflinewidth, opacity=.8](2.5,0) -- (3,0); 
    \draw[draw, blue, double=blue, double distance=2\pgflinewidth, opacity=.8](6,0) -- (6.5,0); 
    \draw (3,3pt) -- (3,-3pt);
    \draw (0,3pt) -- (0,-3pt);
    \draw[densely dotted] (2.5,4pt) -- (2.5,-4pt);
    \draw[densely dotted] (6,  4pt) -- (6,  -4pt);
    \draw[densely dotted] (6.5,4pt) -- (6.5,-4pt);

\draw [decorate,decoration={brace,amplitude=2pt,raise=2ex}]
  (2.5,0) -- (3.2,0) node[midway,yshift=2em]{$\Omega_{b}$};
\draw [decorate,decoration={brace,amplitude=5pt,mirror,raise=2ex}]
  (0,0) -- (3.2,0) node[midway,yshift=-2em]{$\Omega_\local$};
  \draw [decorate,decoration={brace,amplitude=5pt,mirror,raise=2ex}]
  (3,0) -- (6,0) node[midway,yshift=-2em]{$\Omega_\nonlocal$};

\end{tikzpicture}
  \caption{Figure illustrating the domains for the application of maximum principle in the strengthened Cauchy-Schwarz proof. 
   The red highlights correspond to $\Omega_{\interactionSolve}$ and $\localBoundSolve$ while the blue color indicates $\Omega_{\interactionGiven}$ and $\localBoundGiven$.
   Instead of the figures such as \cref{fig:simple-domain}, the $\mathcal O(h)$ overlap is explicitly drawn to show how one can use the maximum principles. 
   }
    \label{fig:max}
\end{figure}
With the continuous case proven, the discrete case follows:
\begin{lemma}[Discrete Strengthened Cauchy-Schwarz inequality]\label{lem:strong-cs-discrete}
    Let \(\vec \theta_{\local} \in \Theta_\local\) and \(\vec \theta_{\nonlocal} \in \Theta_\nonlocal\).
  Then there exists a constant \(0<\overline{c}<1\) and $\bar{h}$ such that for all meshes with discretizations parameters $h \le \bar h$,
  \begin{align*}
    \abs{\left(\vec v_{\local},\vec v_{\nonlocal}\right)_{L^{2}(\Omega_{b})}} &\leq \overline{c} \norm{\vec v_{\local}}_{L^{2}(\Omega_{b})} \norm{\vec v_{\nonlocal}}_{L^{2}(\Omega_{b})}.
  \end{align*}
\end{lemma}
\begin{proof}
  The proof follows from \cite[Lemma A.1]{DEliaBochev2021_FormulationAnalysisComputationOptimization}. 
\end{proof}

We can finally approach the proof of the theorem, which is broken down into three steps: 
\begin{proof}[Proof of \cref{thm:main-theorem}]
We proceed with the proof in several steps. 
First, by \cref{lem:diff-on-overlap}, at any minimum of $\mathcal J$, if any exists, that $\vec u_\nonlocal = \vec u_\local$ on $\Omega_b$.
Furthermore by \cref{lem:equivalent}, that any solution to the optimization problem is equivalent to the splice method and hence related to the the square matrix $\matr A^S$. 
Finally, by \cref{lem:existence-uniqueness} shows that $\matr A^S$ has a null space, meaning there exists an unique solution to the optimization problem. 
\end{proof}

\begin{lemma}\label{lem:diff-on-overlap}
  If $(\vec{\theta}^*_\nonlocal, \vec{\theta}^*_\local)\in \Theta_{\nonlocal}\times\Theta_{\local}$ is a minimizer of \eqref{eqn:discrete-opt}, then $\mathcal J(\vec{\theta}^*_\local, \vec{\theta}^*_\nonlocal) = 0$.
\end{lemma}
\begin{proof}
    Note that $\vec u_L(\vec x_i) = \vec\theta_L$ if $\vec x_i \in \is[\localBoundSolve]$ and $\vec u_N(\vec x_i) = \vec\theta_N$ if $\vec x_i \in \is[\interactionSolve]$ by definition. 
    Thus, we can decompose the summation into the two regions, and rewriting using block matrix notation
    \begin{align*}
      \mathcal{J}\left(
            \vec\theta_L,
            \vec\theta_N 
          \right) &= \frac{1}{2} \int_{\Omega_b} (\vec u_\nonlocal - \vec u_\local)^2 \, dx \\
        &= \frac{1}{2} \begin{bmatrix}
                  \vec u_N(\vec x) - \vec\theta_L \\
       \vec u_L(\vec x) - \vec\theta_N 
   \end{bmatrix}^T \vec M \begin{bmatrix}
                    \vec u_N(\vec x) - \vec\theta_L \\
        \vec u_L(\vec x) - \vec\theta_N 
   \end{bmatrix}
    \end{align*}
    where $\vec M$ is the usual positive definite mass matrix on $\Omega_b$.
    It's implicit the values of $\vec x$ of the top block is over the dofs from $\is[\localBoundSolve]$ and $\is[\interactionSolve]$ for the bottom. 

    Calculating the gradient by using \cref{eqn:ul-opt,eqn:un-opt}, we see that at any minima $(\vec \theta_\local^*, \vec \theta_\nonlocal^*)$
    \begin{align*}
        \nabla \mathcal J \left(
            \vec \theta_L^* ,
            \vec \theta_N^*
        \right) &= 
            \begin{bmatrix}
                     \vec u_N- \vec\theta_L^* \\
       \vec u_L - \vec\theta_N^* 
   \end{bmatrix}^T \vec M \begin{bmatrix}
       -\matr I_\Gamma &  \nabla_{\vec \theta_N} \vec u_N(\vec x) \\
      \nabla_{\vec \theta_L} \vec u_L(\vec x) & -\matr I_I
   \end{bmatrix} \\
   &=  \begin{bmatrix}
                     \vec u_N - \vec\theta_L^* \\
       \vec u_L - \vec\theta_N^* 
   \end{bmatrix}^T \vec M  \begin{bmatrix}
       -\matr I_\Gamma & -\restr[\localBoundSolve]\restr[\nonlocal]^T \lapl[\nonlocal, \nonlocal]^{-1} \lapl[\nonlocal, \interactionSolve] \\
       -\restr[\interactionSolve] \restr[\local]^T\lapl[\local, \local]^{-1} \lapl[\local, \localBoundSolve]  & -\matr I_I
   \end{bmatrix}.
    \end{align*}
    where $\matr I_\Gamma, \matr I_I$ are identity matrices of dimensions $\is[\localBoundSolve], \is[\interactionSolve]$, and the terms $\restr[\localBoundSolve]\restr[\nonlocal]^T$, $\restr[\interactionSolve] \restr[\local]^T$ are restricting to the rows corresponding to $\is[\localBoundSolve], \is[\interactionSolve]$ respectively (seeing as our objective function only evaluates at those points).

    At the minimum, the gradient is zero. 
    If we assume that the first vector is non-zero, then it necessary that the matrix has a nontrivial null-space.
    Suppose that $[\vec{\bar \theta}_L; \vec{\bar \theta}_N]$ is part of the null-space, i.e.
    \begin{align*}
           -\vec{\bar \theta}_L-\restr[\localBoundSolve] \restr[\nonlocal]^T \lapl[\nonlocal, \nonlocal]^{-1} \lapl[\nonlocal, \interactionSolve] \vec{\bar \theta}_N &= \vec 0 \\
         -\restr[\interactionSolve] \restr[\local]^T\lapl[\local, \local]^{-1} \lapl[\local, \localBoundSolve] \vec{\bar \theta}_L  - \vec{\bar \theta}_N&= \vec 0.
    \end{align*}
    Examining \cref{eqn:un-opt,eqn:ul-opt}, we see that the above correspond to the minimization problem with $f = 0$ and homogeneous boundary conditions of $\vec g = 0$.
    However, by \cref{lem:strong-cs-discrete}, this is only possible iff $\vec{\bar \theta}_N = \vec{\bar \theta}_L = 0$, hence contradiction. 
\end{proof}

\begin{remark}
  In order to show the well-posedness for \(\mathbb{P}_{1}-\mathbb{P}_{0}\) coupling, Lemma~\ref{lem:diff-on-overlap} would need to include the case of non-matching discretizations.
\end{remark}

Let $\theta^*_\nonlocal$, $\theta^*_\local$ be any minimizer (possibly not unique) of $\mathcal J$ and let $\vec u_N^* := \vec u_N(\vec\theta_N^*)$, $\vec u_L^* := \vec u_L(\vec\theta_L^*)$ be the solution at the minimizer.
The above lemma implies that $\vec u_N^* = \vec u_L^*$ where they are both defined, meaning we can define $\vec u^O$ with the superscript $\cdot^O$ denoting the optimization solution. 
We can show that $\vec u^O$ corresponds to the spliced solution:
\begin{lemma}\label{lem:equivalent}
    We have that for any minima $\vec u^O$, then $\vec u^S := \vec u^O$ is also a solution to the splice equation and vice versa.
\end{lemma}
\begin{proof}
    For simplicity, let $\vec g = 0$ where we have a homogeneous, zero Dirichlet boundary condition. 
    Given \cref{eqn:un-opt,eqn:ul-opt}, we have that
    \begin{align*}
        \lapl[\nonlocal, \nonlocal] \restr[\nonlocal] \vec u^{O} + \lapl[\nonlocal, (\interactionSolve)] \restr[\interactionSolve] \vec u^{O} &= \vec f_\nonlocal \\
        \lapl[\local, \local] \restr[\local] \vec u^{O} + \lapl[\local, \localBoundSolve] \restr[\localBoundSolve] \vec u^{O} &=  \vec f_\local.
    \end{align*}
    Thus, using \cref{eqn:splice-operator}
        \begin{align*}
        \matr A^S \vec u^O &=
        \restr[\local]^T \begin{bmatrix}
        \lapl[\local, \local] & \lapl[\local, \localBoundSolve] 
      \end{bmatrix} 
      \begin{bmatrix}
        \restr[\local] \\
        \restr[\localBoundSolve]
    \end{bmatrix}\vec u^O + 
        \restr[\nonlocal]^T \begin{bmatrix}
        \lapl[\nonlocal, \nonlocal] & \lapl[\nonlocal, \interactionSolve]
    \end{bmatrix}  \begin{bmatrix}
        \restr[\nonlocal] \\
        \restr[\interactionSolve]
    \end{bmatrix}\vec u^O \\
     &= \restr[\local]^T
        (\lapl[\local, \local]\restr[\local] + \lapl[\local, \localBoundSolve] \restr[\localBoundSolve])\vec u^O + 
        \restr[\nonlocal]^T
        (\lapl[\nonlocal, \nonlocal] \restr[\nonlocal] + \lapl[\nonlocal, (\interactionSolve)]\restr[\interactionSolve])\vec u^O \\
        &= \restr[\local]^T \vec f_\local + \restr[\nonlocal]^T \vec f_\nonlocal
    \end{align*}
    meaning that the optimization with $\mathcal{O}(h)$ overlap is equivalent to the splicing method. 
\end{proof}
\begin{lemma}\label{lem:existence-uniqueness}
  The splice coupling has a unique solution.
\end{lemma}
\begin{proof}
  We need to show that the nullspace of \(\matr{A}^{S}\) only consists of the zero vector.

  Using the strengthened Cauchy-Schwarz inequality of \cref{lem:strong-cs-discrete}, we obtain
  \begin{align*}
    \mathcal{J}(\vec \theta_{\local},\vec  \theta_{\nonlocal}) &= \norm{\vec u_{\local}}_{L^{2}(\Omega_{b})}^{2} + \norm{\vec u_{\nonlocal}}_{L^{2}(\Omega_{b})}^{2} - 2\left(\vec u_{\local},\vec u_{\nonlocal}\right)_{L^{2}(\Omega_{b})} \\
    &\geq \norm{\vec u_{\local}}_{L^{2}(\Omega_{b})}^{2} + \norm{\vec u_{\nonlocal}}_{L^{2}(\Omega_{b})}^{2} - 2c\norm{\vec u_{\local}}_{L^{2}(\Omega_{b})} \norm{\vec u_{\nonlocal}}_{L^{2}(\Omega_{b})} \\
    &\geq (1-c) \left(\norm{\vec u_{\local}}_{L^{2}(\Omega_{b})}^{2} + \norm{\vec u_{\nonlocal}}_{L^{2}(\Omega_{b})}^{2}\right).
  \end{align*}
  Therefore \(\mathcal{J}=0\) implies that \(\vec u_{\local} \) and \(\vec u_{\nonlocal} \) are zero on \(\Omega_{b}\) and in turn that \(\vec \theta_{\local}\) and \(\vec \theta_{\nonlocal}\) are zero.
\end{proof}

\bibliography{sources.bib}{}
\bibliographystyle{plain}

\end{document}